\documentclass[a4paper,10pt]{article}
\usepackage[utf8]{inputenc}

\usepackage{amscd,  graphicx, color, mathrsfs}
\usepackage{lineno}
\usepackage{dsfont}

\usepackage{amsmath, amsthm, amssymb, amsfonts, enumerate}
\usepackage[colorlinks=true,linkcolor=blue,urlcolor=blue]{hyperref}

\usepackage{amsfonts,latexsym,amstext}
\usepackage[active]{srcltx}
\usepackage{xcolor}
\usepackage{amsmath,amsfonts,calrsfs, amssymb,color}
\usepackage{amsthm}
\usepackage{amssymb}

\newtheorem{theorem}{Theorem}[section]
\newtheorem{lemma}[theorem]{Lemma}
\newtheorem{proposition}[theorem]{Proposition}
\newtheorem{definition}[theorem]{Definition}
\newtheorem{example}[theorem]{Example}

\newtheorem{remark}[theorem]{Remark}
\newtheorem{corollary}[theorem]{Corollary}

\numberwithin{equation}{section}

\def\qed{{\hfill\hbox{\enspace${ \square}$}} \smallskip}
\def\sqr#1#2{{\vcenter{\vbox{\hrule height .#2pt \hbox{\vrule
 width .#2pt height#1pt \kern#1pt \vrule
width .#2pt} \hrule height .#2pt}}}}
\def\square{\mathchoice\sqr54\sqr54\sqr{4.1}3\sqr{3.5}3}

\def\ds{\begin{displaystyle}}
\def\eds{\end{displaystyle}}

\def\<{\langle }
\def\>{\rangle }

\def\R{\mathbb R}
\def\N{\mathbb N}

\allowdisplaybreaks

\title{Bi-Kolmogorov type operators and weighted Rellich's inequalities \footnote{The authors are members of the Gruppo Nazionale per l’Analisi Matematica, la Probabilità e le loro Applicazioni (GNAMPA) of the Istituto Nazionale di Alta Matematica (INdAM).}}

\usepackage{authblk}

\author[1]{Davide ADDONA\thanks{davide.addona@unipr.it}
} 
\author[2]{Federica GREGORIO\thanks{fgregorio@unisa.it}} 
\author[3]{Abdelaziz RHANDI\thanks{arhandi@unisa.it}} 
\author[2]{Cristian TACELLI\thanks{ctacelli@unisa.it}} 

\affil[1]{Dipartimento di Matematica, Fisica e Informatica, Universit\`a di Parma, Parma, Italy}

\affil[2]{Dipartimento di Matematica, Universit\`a di Salerno, Salerno, Italy} 
 \affil[3]{Dipartimento di Ingegneria dell'Informazione ed Elettrica e Matematica Applicata, Universit\`a di Salerno, Salerno, Italy}
\date{}

\begin{document}

\maketitle

\begin{abstract}
In this paper we consider the symmetric Kolmogorov operator $L=\Delta +\frac{\nabla \mu}{\mu}\cdot \nabla$ on $L^2(\R^N,d\mu)$, where $\mu$ is the density of a probability measure on $\R^N$. Under general conditions on $\mu$
we prove first weighted Rellich's inequalities with optimal constants and deduce that the operators $L$ and $-L^2$ with domain $H^2(\R^N,d\mu)$ and $H^4(\R^N,d\mu)$ respectively, generate analytic semigroups of contractions on $L^2(\R^N,d\mu)$. We observe that $d\mu$ is the unique invariant measure for the semigroup generated by $-L^2$ and as a consequence we describe the asymptotic behaviour of such semigroup and obtain some local positivity properties. As an application we study the bi-Ornstein-Uhlenbeck operator and its semigroup on $L^2(\R^N,d\mu)$.
\end{abstract}

\medskip


\section{Introduction}

The focus of this article is to study some parabolic properties in weighted $L^2$ spaces for a class of forth order operators with unbounded coefficients:   bi-Kolmogorov type operators. We define our operators $(A,D(A))$ as   operators associated to the sesquilinear form
\[a_L(u,v)=\int_{\R^N}Lu\overline{Lv}\,d\mu\qquad u,v\in D(a_L):=D(L),\]
where $L$ is a second order Kolmogorov type operator defined on smooth functions as $Lf=\Delta f+\frac{\nabla\mu}{\mu}\cdot\nabla f$ and $d\mu$ is some probability measure which has density $\mu$ with respect to the Lebesgue measure. Second order Kolmogorov type operators have been widely studied in the literature,  the prototype  being the Ornstein-Ulhenbeck operator corresponding to the choice of   a Gaussian measure $\mu$, see for example \cite{Lunardi, MPRS, lor-16, ber-lor, pru-rha-sch}. The Gaussian measure is the unique invariant measure for the Ornstein-Ulhenbeck semigroup and an explicit formula for the semigroup is available; having such a formula simplifies the study of the main properties
of the semigroup.  It is known that the most appropriate space  to study parabolic properties of the Ornstein-Ulhenbeck semigroup  is the weighted  space  $L^2_\mu(\R^N):=L^2(\R^N,d\mu)$  where it gives rise to  a strongly continuous analytic semigroup (the result holding true in $L^p_\mu(\R^N)$ for any $p\in[1,\infty)$). Roughly speaking, the Ornstein-Ulhenbeck operator in $L^2_\mu(\R^N)$ is the counterpart of the Laplacian operator in $L^2(\R^N)$.

Having defined the bi-Kolmogorov type operator through the form $a_L$, in order to establish generation results in $L^2_\mu(\R^N)$ for $A$, we need generation results for the second order Kolmogorov type operator $L$. In \cite{alb-lor-man} under mild assumptions on $\mu$, the authors prove that the closure of $(L, C_c^\infty(\R^N))$ on $L^2_\mu(\R^N)$ is the generator of an analytic contraction $C_0$-semigroup of angle $\frac{\pi}{2}$ on $L^2_\mu(\R^N)$. In particular they assume 
\textbf{Hypothesis} $(H1)$
\begin{itemize}
\item[(i)]  $\mu\in H^1_{\rm loc}(\R^N)$, $\frac{\nabla\mu}{\mu}\in L^r_{\rm loc}(\R^N,\R^N)$ for some $r>N$,
\item[(ii)] $\inf_{x\in K}\mu(x)>0$ for every compact $K\subset \R^N$. 
\end{itemize}
Therefore, we are able to state that, under Hypothesis $(H1)$, the operator  $(-A,D(A))$ generates an analytic contraction $C_0$-semigroup of angle $\frac{\pi}{2}$ on $L^2_\mu(\R^N)$.

In this paper we also deal with perturbation of the operator $A$ by a singular potential. It is known that there is a strong relation between second order Schr\"odinger type operators and Hardy's inequalities. When one deals with a complete operator, i.e., allowing for a drift term, some generalised Hardy inequality is needed. In \cite{can-gre-rha-tac} suitable conditions on $\mu$ are obtained for the validity of a weighted Hardy inequality of the following form 
\[C_0\int_{\R^N}\frac{|u(x)|^2}{|x|^2}d\mu\leq\int_{\R^N}|\nabla u(x)|^2d\mu+C\int_{\R^N}|u(x)|^2d\mu,\qquad u\in H^1_\mu(\R^N)\]
where $C_0=\left(\frac{N-2}{2}\right)^2$ is the best constant in Hardy's inequality and $C$ is a positive constant. 
As a consequence existence of positive solutions to the  parabolic
problem associated to the perturbed operator $L+V$ with $0\leq V(x)\leq\frac{C_0}{|x|^2}$ is stated. 
The bi-Laplacian operator perturbed by the inverse fourth order potential has been studied in \cite{gre-mil}: in this case a Rellich inequality is needed. We prove here a weighted Rellich inequality with respect to the operator $L$. Indeed, under the assumptions \\
\textbf{Hypothesis} $(H2)$ 
\begin{itemize}
\item[(i)] $\sqrt\mu\in H^1_{\rm loc}(\R^N),\Delta\mu\in L^1_{\rm loc}(\R^N)$,
\item[(ii)] there exists a $R_0>0$ such that $ |x|^2\left(\frac14\left|\frac{\nabla\mu}{\mu}\right|^2-\frac12\frac{\Delta\mu}{\mu}\right)\leq\frac14\frac{1}{|\log|x||^2}\quad\forall\,x\in B_{R_0},$
\item[(iii)] $\frac14\left|\frac{\nabla\mu}{\mu}\right|^2-\frac12\frac{\Delta\mu}{\mu}$ is bounded from above in $\R^N\setminus B_R$ for every $R>0$,
\end{itemize}
we prove that there exists $C>0$ such that the following inequality with optimal constant holds 
\[(C_0-1)^2\int_{\R^N}\frac{|u|^2}{|x|^4}d\mu\leq\int_{\R^N}|L u|^2d\mu+C||u||^2_{H^1_\mu(\R^N)}\]
for any $u\in H^2_\mu(\R^N)$, $N\geq5$.
This inequality allows us to establish generation results for the perturbed operator $-A+V$ when $0\leq V\leq\frac{(C_0-1)^2}{|x|^4}$.

An important feature of the Ornstein-Uhlenbeck semigroup in $L^2_\mu(\R^N)$ (and also in $L^p_\mu(\R^N)$ for  $p\in(1,\infty)$) is that a complete characterization of its   generator is available. In particular, it is known that the domain of the Ornstein-Ulhenbeck operator in $L^2_\mu$ is exactly the weighted Sobolev space $H^2_\mu(\R^N)$, cf. \cite{dap,Lunardi}. The same result holds when $\mu$ takes the form $\mu=e^{-\phi}$ with $\phi$ convex or $\phi\in C^3(\R^N)$ and $\phi$ together with its derivatives up to second order have polynomial growth, cf. \cite{dap-ves}. In this paper we provide more general assumptions that imply that the domain of the Kolmogorov operator $L$ coincides with $H^2_\mu(\R^N)$. In particular, as a consequence of weighted Hardy's and Rellich's inequalities, we derive  some useful estimates   such as a weighted interpolation inequality and a kind of Calderon-Zygmund inequality that allows us to deduce that $D(L)=H^2_\mu(\R^N)$ if $N\geq5$, Hypotheses $(H1)$ and $(H2)$ hold, and additionally \\
\textbf{Hypothesis $(H3)$}
(i) $\mu \in W^{2,1}_{\rm loc}(\R^{N})$  and $\left |D_{i}\left (\frac{D_{j}\mu}{\mu}\right)\right |\leq \frac{\varepsilon}{|x|^{2}}+C\left |\frac{\nabla \mu}{\mu}\right|\,\forall\,i,j=1,\dots,N$.\\
Moreover, through more general higher order weighted Rellich's inequalities, we are also able to characterize the domain of the bi-Kolmogorov operator $A$. We show that,  assuming further $N\geq7$ and\\ \textbf{Hypothesis $(H3)$}
(ii) $\mu \in W^{3,1}_{\rm loc}(\R^{N})$ and $\left |D_{ij}\left (\frac{D_{k}\mu}{\mu}\right)\right |\leq \frac{\varepsilon}{|x|^{3}}+C\left |\frac{\nabla \mu}{\mu}\right|\,\forall\,i,j,k=1,\dots,N$\\
then $D(A)=H^4_\mu(\R^N).$

The interest in higher order operators has grown considerably in recent times due to their applications in many fields of science, for example they are involved in models of elasticity \cite{mel} or condensation in graphene \cite{sed}, free boundary problems \cite{ada} and non-linear elasticity \cite{ant, skr}. Among the best known features of higher order differential equations is the failure of maximum principles, and hence of the positivity preserving property of the semigroup. The heat kernel of the bi-Laplacian operator was studied by Hochberg already in 80s, cf. \cite{hoc}. Together with a stocastic interpretation for the underlying process, he proves that the integral kernel  $k(t,x,y)$ satisfies
$$k(t,x,y) \approx   M t^{-\frac{1}{6}}|x-y|^{-\frac{1}{3}}\exp\left(-\frac{3}{8}\left(\frac{|x-y|^4}{4t}\right)^{\frac{1}{3}}\right)
\cos\left(\frac{3\sqrt3}{8}\left(\frac{|x-y|^4}{4t}\right)^{\frac{1}{3}}\right)
$$
for large $x$,  $M$ is a positive constant  and the approximation holds up to lower-order terms. This shows in particular that $k$ has an oscillatory character and changes sign infinitely often. However, even if the classical notion of positivity fails, one can ask for a relaxation of this property:  \emph{eventual positivity}, meaning that, considering positive initial data, the solution to a Cauchy problem may become positive for large enough time. A further relaxation is the property of \emph{local eventual positivity}, i.e., eventual positivity on compact sets.  Relying on the explicit formula of the kernel on $\R^N$, in \cite{gaz-gru} it is proved that for continuous, compactly supported, positive initial data $u_0$, the solution to the Cauchy problem associated to the bi-Laplacian in $\R^N$ is \emph{individually locally eventually positive} meaning that: for any compact set $K\subset\R^N$, there exists $T_K>0$ that depends on $u_0$ such that $u(t,x)>0$ for all $t\geq T_K$ and $x\in K$; and there exists $\tau>0$ 
that depends on $u_0$ such that for any $t>\tau$ there exists a $x_t\in\R^N$ such that $u(t,x_t)<0$. A generalisation in \cite{fer-gaz-gru} covers the case of initial data that decay fast at infinity and in \cite{fer} it is shown that also fractional polyharmonic equations display the same behaviour. In Section \ref{bi-OU}, we obtain a semi-explicit formula for the kernel of the bi-Ornstein-Ulhenbeck operator showing also in this case an oscillatory character that prevents the kernel to be positive. However, we are able to state suitable  conditions on the measure $\mu$ that imply local eventual positivity for the semigroup $(e^{-tA})_{t\geq0}$. To this purpose  we rely on an abstract theory initiated by Daners, Gl\"uck and Kennedy in \cite{dan-glu-ken1,dan-glu-ken2}. They state sufficient conditions on the generators in order to obtain eventual positivity for the associated semigroups. The theory has been developed in \cite{aro} considering milder conditions which imply local eventual positivity.  Firstly, investigating on the asymptotic behaviour of $(e^{-tA})_{t\geq0}$, we  prove that the semigroup is asymptotically irreducible, a rather strong property. Further, we give sufficient conditions on the measure $\mu$ for both individual and uniform  local  eventual  positivity of  the bi-Kolmogorov semigroup $\left(e^{-tA}\right)_{t\geq0}$.  

The paper is organised as follows. In Section \ref{operator} we introduce the fourth order operator $A$ and state generation results. In Section \ref{propsem} we prove that $\mu$ is the unique invariant measure for $(e^{-tA})_{t\geq0}$ and we study asymptotic and positivity properties for the semigroup. Section \ref{Rel-ineq} deals with the proof of the weighted Rellich inequality and the optimality of the constant. Some more estimates and further higher order weighted Rellich's inequality of Section \ref{domain-bi-kolm} serve for the characterisation of both the domains of $L$ and $A$. Finally, in Section \ref{bi-OU} we adapt all the obtained results to the bi-Ornstein-Ulhenbeck semigroup and give a semi-explicit formula for its kernel.

\section{The bi-Kolmogorov operator}\label{operator}

We consider the square of a general Kolmogorov-type operator. To this purpose let us consider some probability measure $d\mu$ which has density $0\le \mu\in L^1(\R^N)$ with respect to the Lebesgue measure. We assume the following on $\mu$\\
\textbf{Hypothesis} $(H1)$
\begin{itemize}
\item[(i)]  $\mu\in H^1_{\rm loc}(\R^N)$, $\frac{\nabla\mu}{\mu}\in L^r_{\rm loc}(\R^N,\R^N)$ for some $r>N$,
\item[(ii)] $\inf_{x\in K}\mu(x)>0$ for every compact $K\subset \R^N$. 
\end{itemize}

We consider the operator $L$ defined on smooth functions $f$ as
\begin{align*}
Lf=\Delta f+\frac{\nabla \mu}{\mu}\cdot \nabla f.
\end{align*}
By \cite[Corollary 3.7]{alb-lor-man}, the closure of $(L,C_{c}^{\infty}(\R^N))$ on $L^2_\mu(\R^N)$ denoted by $(L,D(L))$ generates an analytic contraction $C_0$-semigroup of angle $\frac{\pi}{2}$ on $L^2_\mu(\R^N)$. We notice that 
$(L,D(L))$ coincides with the operator associated to the sesquilinear form
\begin{equation}\label{form-L}
a(u,v):=\int_{\R^N}\nabla u\cdot \overline{\nabla v}\,d\mu,\quad u,\,v\in H^1_\mu(\R^N),
\end{equation}
since it coincides with $L$ on $C_c^\infty(\R^N)$ and both are generators on $L^2_\mu(\R^N)$. In particular, $D(L)\subset H^1_\mu(\R^N)$ and
\begin{align*}
\int_{\R^N}Lf\,d\mu=
-\int_{\R^N}  \nabla f\cdot \nabla 1 \, d\mu=0,
\end{align*}
for all $f\in D(L)$.
This implies that $\mu$ is a symmetrizing invariant measure for the semigroup generated by $L$, and so we have the following.
\begin{remark}\label{rem:2-1}
It follows from \cite[Corollary 2.10]{bog_kry_roc01}  that  $\mu\in W^{1,r}_{\rm loc}(\R^N)$ for some $r>N$ and by Sobolev's embedding we deduce that $\mu\in C^{1-\frac{N}{r}}_{\rm loc}(\R^N)$. In particular, this implies that
for any compact set $K\subset \R^N$ there exists a positive constant $c=c_K$ such that $c^{-1}\leq \mu(x)\leq c$ for any $x\in K$, where the lower bound follows from the assumption that $\inf_{K}\mu>0$ for any compact set $K\subset \R^N$, and the upper bound is a byproduct of the partnership of $\mu$ to $C^{1-{\frac{N}{r}}}_{\rm loc}(\R^N)$.
\end{remark}

We introduce on the domain of  $L$  the sesquilinear form
\begin{align}\label{form-aL} 
a_L(u,v):=\int_{\R^N}Lu\ \overline{Lv}\ \!d\mu, \quad u,v\in D(a_L):=D(L).
\end{align}
Since $a_L$ is positive semidefinite, there exists an operator $(A,D(A))$ which satisfies
\begin{align*}
a_L(u,v)= \int_{\R^N}A \ \!u \ \overline v\ \!d\mu, \quad u\in D(A), \ v\in D(L).
\end{align*}
 Further,  since $C_c^\infty(\R^N)$ is a core for $L$ in $L^2_\mu(\R^N)$ it follows that $a_L$ is densely defined. It is simple to check that $a_L$ is continuous and closed. Therefore, $-A$ is the generator of an analytic contraction $C_0$-semigroup on $L^2_\mu(\R^N)$.

Let us compute the explicit form of the operator $ A$. We denote by ${\rm div}_\mu$ the adjoint operator of $\nabla$ in $L^2_\mu(\R^N,\R^N)$. Let $\Phi$ be a smooth vector field. Then, for any $f\in C^1_c(\R^N)$, the integration by parts formula gives
\begin{align*}
\int_{\R^N}\nabla f\cdot \Phi \ \!d\mu 
= &- \int_{\R^N}f\left({\rm  div}\Phi+\frac{ \nabla \mu\cdot\Phi }{\mu}\right)d\mu, 
\end{align*}
where ${\rm div}$ is the classical divergence operator. It follows that ${\rm div}_\mu=-\left({\rm div}+\frac{\nabla \mu}{\mu}\right)$ and it is easy to see that $L={\rm div}_\mu \nabla$. Take now $u,\,v\in C_c^\infty(\R^N)$ and assume that $\mu\in W^{2,1}_{\rm loc}(\R^N)$.
Then, two integration by parts give
\begin{align}
\label{int_parti_1}
a_L(u,v)
=& \int_{\R^N}
Lu\ \overline{Lv}\ \!d\mu 
=  \int_{\R^N}
L\ \!u({\rm div}_\mu \nabla\overline v)d\mu
=- \int_{\R^N}\nabla(L\ \!u)\cdot \nabla\overline v\ \!d\mu 
= \int_{\R^N}{\rm div}_\mu\nabla(Lu)\overline v\ \!d\mu .
 \end{align}
We have
\begin{align}
\label{derivata_1}
\nabla(L\ \!u)\cdot \nabla v
= & \sum_{i=1}^N\partial_i\left(\Delta u +\frac{\nabla \mu}{\mu}\cdot \nabla u\right)\partial_iv
= \sum_{i,j=1}^N\left(\partial^3_{ijj}u+ \frac{\partial_{ij}^2\mu\partial_j u}{\mu}+\frac{\partial_{j}\mu\partial_{ij}^2 u}{\mu}-\frac{\partial_{i}\mu\partial_j\mu \partial_j u}{\mu^2}\right)\partial_iv.
\end{align}
Let us assume now that $\mu\in W^{3,1}_{\rm loc}(\R^N)$ and consider the vector field
\begin{align*}
\Phi:=(\Phi_1,\ldots,\Phi_N), \quad \Phi_i:=\partial_i\left(\Delta u +\frac{\nabla \mu}{\mu}\cdot \nabla u\right), \quad i=1,\ldots,N,\,u\in C_c^\infty(\R^N).
\end{align*}
It remains to compute ${\rm div_\mu}\Phi$. By taking advantage from \eqref{derivata_1} we infer that
\begin{align*}
{\rm div_\mu}\Phi
= \sum_{i=1}^N\left(\partial_i\Phi_i+\frac{\partial_i\mu\Phi_i}{\mu}\right).
\end{align*}
For any $i=1,\ldots,N$ we have
\begin{align}
\label{derivata_2}
\partial_i\Phi_i
= &\sum_{j=1}^N\bigg(  \partial^4_{iijj}u+\frac{\partial^3_{iij}\mu\partial_ju}{\mu}+\frac{\partial_{ij}^2\mu\partial^2_{ij} u}{\mu}-\frac{\partial_{ij}^2\mu\partial_i\mu\partial_j u}{\mu^2}+\frac{\partial_{ij}^2\mu\partial^2_{ij} u}{\mu}+\frac{\partial_{j}\mu\partial_{iij}^3 u}{\mu}-\frac{\partial_{j}\mu\partial_i\mu \partial_{ij}^2 u}{\mu^2} \notag \\
& -\frac{\partial_{ii}^2\mu\partial_j\mu \partial_j u}{\mu^2}-\frac{\partial_{i}\mu\partial_{ij}^2\mu \partial_j u}{\mu^2}-\frac{\partial_{i}\mu\partial_j\mu \partial_{ij}^2 u}{\mu^2}+2\frac{(\partial_{i}\mu)^2\partial_j\mu \partial_j u}{\mu^3}\bigg).
\end{align}
Further,
\begin{align}
\label{derivata_3}
\frac{\partial_i\mu\Phi_i}{\mu}
=\sum_{j=1}^N\bigg( \frac{\partial_i\mu\partial^3_{ijj}u}{\mu}+ \frac{\partial_{ij}^2\mu\partial_i\mu\partial_j u}{\mu^2}+\frac{\partial_{j}\mu\partial_i\mu\partial_{ij}^2 u}{\mu^2}-\frac{(\partial_{i}\mu)^2\partial_j\mu \partial_j u}{\mu^3}\bigg).
\end{align}
Putting together \eqref{derivata_2} and \eqref{derivata_3} we get
\begin{align*}
{\rm div}_\mu\nabla(Lu)
= & \sum_{i,j=1}^N
\bigg( 
\partial_{iijj}^4u+2\frac{\partial_j\mu \partial_{iij}^3u}{\mu}+2\frac{\partial^2_{ij}\mu \partial_{ij}^2u}{\mu}
-\frac{\partial_j\mu\partial_i\mu\partial_{ij}^2u}{\mu^2}
-\frac{\partial_{ij}^2\mu\partial_j\mu\partial_i u}{\mu^2} \notag \\
&  -\frac{\partial_{iij}^3\mu\partial_j u}{\mu}-\frac{\partial_{ii}^2\mu\partial_j\mu\partial_ju }{\mu^2}+\frac{(\partial_i\mu)^2\partial_j\mu\partial_ju}{\mu^3}\bigg) \end{align*}
and then, for $u\in C_c^\infty(\R^N)$, we have
\begin{align}
 Au
= & \Delta^2u+2\frac{\nabla \mu}{\mu}\cdot {\nabla(\Delta u)}+2\frac{{\rm Tr}[D^2\mu D^2u]}{\mu}
-\left(D^2u\frac{\nabla\mu}{\mu}\right)\cdot\frac{\nabla\mu}{\mu}
- \left(D^2\mu\nabla u \right)\cdot\frac{\nabla\mu}{\mu} \notag\\
& +\frac{\nabla(\Delta\mu)}{\mu}\cdot\nabla u-\frac{\Delta\mu}{\mu}\frac{\nabla\mu}{\mu}\cdot\nabla u+\left|\frac{\nabla\mu}{\mu}\right|^2\frac{\nabla\mu}{\mu}\cdot\nabla u.
\label{explicit_formula_A}
\end{align}
Thus, we have the following result.
\begin{proposition}
Assume $(H1)$ is satisfied. Then
the operator $-A$ associated to the sesquilinear form $a_L$ defined by \eqref{form-aL} generates an analytic contraction $C_0$-semigroup $e^{-tA}$ of angle $\frac{\pi}{2}$ on $L_\mu^2(\R^N)$. If in addition $\mu \in W^{3,1}_{\rm loc}(\R^N)$ then $C_c^\infty(\R^N)\subset D( A)$ and $ A$ is given by \eqref{explicit_formula_A}.
\end{proposition}

\section{Asymptotic properties of bi-Kolmogorov semigroups}\label{propsem}

We start with some considerations on the operator $L$ introduced above. Let us notice that $D(L)\subset H^1_\mu(\R^N)$, and that if $g\in D(L)$ satisfies $\nabla g=0$ $\mu$-a.e. in $\R^N$ it follows that $g$ is constant $\mu$-a.e. in $\R^N$.
\begin{lemma}
\label{lemma:propr_L_spectrum}
Let $L$ be as above. Then:
\begin{itemize}
\item[(i)] $1\in D(L)$ and $L1=0$. This means that $0$ is an eigenvalue of $L$ and the constant functions belong to the associated eigenspace;
\item [(ii)] the eigenspace of $L$ associated to $0$ coincides with the constant functions, and it equals the fixed points of the semigroup $(e^{tL})_{t\geq0}$, i.e., the set
\begin{align*}
\mathcal E:=\{f\in L^2_\mu(\R^N):e^{tL}f=f \ \mu\textup{-a.e. in }\R^N\}.
\end{align*}
\end{itemize}
\end{lemma}
\begin{proof}
{$ {(i)}$}  Consider a function $\eta\in C_c^\infty(\R^N)$ satisfying $0\le \eta \le 1,\,\eta \equiv 1$ in $B_1$ and $\eta \equiv 0$ in $\R^N\setminus B_2$, where $B_R$ denotes the ball with centre $0$ and radius $R$. Then one can see easily that $\eta_n$ converges to $1$ and $L\eta_n$ converges to $0$, where $\eta_n(x):=\eta(x/n)$ for $x\in \R^N$ and $n\in \mathbb{N}$. Since $C_c^\infty(\R^N)$ is a core for $L$, it follows that $1\in D(L)$ and $L1=0$. The linearity of $L$ gives the second part.

\vspace{3mm}
{$ {(ii)}$}  Let $f\in D(L)$ belong to the eigenspace associated to $0$. Hence,
\begin{align*}
0=\int_{\R^N}Lf \ \!fd\mu=\int_{\R^N}|\nabla f|^2d\mu,
\end{align*}
which gives $\nabla f=0$ $\mu$-a.e. in $\R^N$. The above considerations implies that $f$ is constant $\mu$-a.e. in $\R^N$. Let us prove that $\mathcal E$ coincides with the eigenspace of $L$ associated to $0$. The inclusion $\subset$ is trivial, indeed if $f\in \mathcal E$ then from the definition of $L$ it follows that $f\in D(L)$ and $Lf=0$. This implies that $f$ is a constant function. To prove the converse inclusion, we recall that  $L$ generates a strongly continuous semigroup, and so
\begin{align}
\label{form_int_c0_smgr}
e^{tL}f-f=L\int_0^te^{sL}fds, \quad f\in L^2_\mu(\R^N), \ t\geq0.
\end{align}
Let us assume that $f\in D(L)$ and $Lf=0$. Hence, $Le^{sL}f=e^{sL}Lf=0$ for any $s>0$, and from \eqref{form_int_c0_smgr} it follows that
\begin{align*}
e^{tL}f-f=L\int_0^te^{sL}fds=\int_0^tLe^{sL}fds=0, \quad t\geq0,
\end{align*}
which gives $f\in\mathcal E$.
\end{proof}

\subsection{Analysis of \texorpdfstring{$(e^{-tA})_{t\geq0}$}{etAt>0}}
We prove that the  measure $\mu$ is an invariant measure for the  semigroup generated by $-A$. We recall first the definition of invariant measures.
\begin{definition}
Let $\nu$ be a probability Borel measure, and let $S(t)$ be a $C_0$-semigroup of bounded linear operators on $L^2_\nu(\R^N)$. We say that $\nu$ is an invariant measure for $S(t)$ if
\begin{align*}
\int_{\R^N}S(t)fd\nu=\int_{\R^N}fd\nu, \quad f\in C_b(\R^N), \ t\geq0.
\end{align*}
\end{definition}

The following result shows that $\mu$ is an invariant measure for the semigroup $(e^{-tA})_{t\geq0}$ generated by $-A$.
\begin{proposition}
\label{pro:spectralA}
$\mu$ is an invariant measure for $(e^{-tA})_{t\geq0}$. Further, $0$ is an eigenvalue of $A$, and the corresponding eigenspace consists of constant functions.
\end{proposition}
\begin{proof}
By density, it is enough to prove that
\begin{align*}
\int_{\R^N}e^{-tA}fd\mu=\int_{\R^N}fd\mu, \quad f\in D( A).
\end{align*}
Let $f\in D( A)$, then we have
\begin{align}
\label{teo:fond_calcolo}
e^{-tA}f-f=\int_0^te^{-sA}{(-A)}fds, \quad t\geq0, 
\end{align}
where the equality is meant in $L^2_\mu(\R^N)$. By integrating both  sides of \eqref{teo:fond_calcolo} on $\R^N$ with respect to $\mu$, by applying Fubini theorem and by using the fact that $e^{-sA}$ is a bounded linear operator for any $s\in[0,T]$ we get
\begin{align}
\label{teo_fon_cal_int}
\int_{\R^N}(e^{-tA}f-f)d\mu=-\int_0^te^{-sA}\left(\int_{\R^N}Afd\mu\right)ds, \quad t\geq0.
\end{align}
We now claim that
\begin{align*}
\int_{\R^N}Agd\mu=0, \quad g\in D(A).
\end{align*}
If the claim is true, then the right-hand side of \eqref{teo_fon_cal_int} vanishes and we get the thesis. From Lemma \ref{lemma:propr_L_spectrum} we know that the constant function equal to $1$ is in $D(L)$ and $L1=0$. Hence,
\begin{align*}
\int_{\R^N}Agd\mu=\int_{\R^N}Ag \ \!1d\mu 
= \int_{\R^N}LgL1d\mu=0, \quad g\in D(A).
\end{align*} 
Therefore, the claim is  proved and the thesis follows at once.

Let us prove the second part of the statement. Since $1\in D(L)$ and $L1=0$ it follows that $1\in D(A)$ and for any constant function $f$, $f\in D(A)$ and $Af=0$. Hence, $0$ is an eigenvalue of $A$ and $\R$ is contained in the corresponding eigenspace. Let us prove that any function $f\in L^2_\mu(\R^N)$ such that $Af=0$ $\mu$-a.e. is constant. To this aim, let $f\in L^2_\mu(\R^N)$ be such that $Af=0$ $\mu$-a.e. Then,
\begin{align*}
0=\int_{\R^N}Af\ \!fd\mu=\int_{\R^N}(Lf)^2d\mu.
\end{align*}
This implies that $Lf=0$ $\mu$-a.e. Lemma \ref{lemma:propr_L_spectrum}$(ii)$ allows us to conclude.
\end{proof}

Now we show that $\mu$ is ergodic with respect to the semigroup $(e^{-tA})_{t\geq0}$, i.e.,
\begin{align*}
L^2-\lim_{t\rightarrow+\infty}\frac{1}{t}\int_0^te^{-sA}fds=\int_{\R^N}fd\mu, \quad f\in L^2_\mu(\R^N).
\end{align*}
\begin{proposition}
$\mu$ is ergodic with respect to the semigroup $(e^{-tA})_{t\geq0}$. As a byproduct, $\mu$ is the unique invariant measure for $(e^{-tA})_{t\geq0}$.
\end{proposition}
\begin{proof}
Let us set $P_tf:=t^{-1}\int_0^te^{-sA}fds$ for any $f\in L^2_\mu(\R^N)$, arguing as in the proof of \cite[Proposition 8.1.13]{ber-lor} we infer that there exists an operator $P_\infty$  
such that 
\begin{align*}
L^2-\lim_{t\rightarrow+\infty}P_tf=P_\infty f, \quad f\in  L^2_\mu(\R^N),
\end{align*}
that $e^{-rA}  P_\infty=P_\infty$ for any $r>0$ and that $P_\infty$ is a projection on the space $C:=\{f\in L^2_\mu(\R^N):e^{-tA}f=f\ \mu{\textup{-a.e. in }}X\}$. Let us show that $C$ only consists of constant functions. If $f\in C$ then $f\in D(A)$ and $Af=0$. From Proposition \ref{pro:spectralA} it follows that $f$ is a constant function. On the other hand, if $f$ is constant then $Af=0$, and from \eqref{teo:fond_calcolo} we infer that $e^{-tA}f-f=0$ for $\mu$-a.e. in $\R^N$. This implies that the dimension of $C$ is $1$, and therefore there exists a linear operator $S\in (L^2_\mu(\R^N))^*$ such that
$P_\infty f=S(f)$ for any $f\in L^2_\mu(\R^N)$. From the Riesz representation theorem there exists a function $g\in L^2_\mu(\R^N)$ such that
\begin{align*}
S(f)=\int_{\R^N}fgd\mu, \quad f\in L^2_\mu(\R^N).
\end{align*}
Let us prove that $g=1$ for $\mu$-a.e. in $\R^N$. Integrating $P_tf$ on $\R^N$ with respect to $\mu$, by applying Fubini's theorem and recalling that $\mu$ is an invariant measure for $e^{-tA}$ we get
\begin{align*}
\int_{\R^N}P_tfd\mu
=\int_{\R^N}\frac{1}{t}\left(\int_0^t e^{-sA}fds\right)d\mu
= \frac{1}{t}\int_0^t\left(\int_{\R^N}e^{-sA}fd\mu\right)dt
=\int_{\R^N}fd\mu.
\end{align*}
Letting $t\rightarrow+\infty$ it follows that
\begin{align*}
\int_{\R^N}fd\mu
= & \int_{\R^N}S(f)d\mu
= \int_{\R^N}fgd\mu, \quad f\in L^2_\mu(\R^N),
\end{align*}
which gives $g=1$ for $\mu$-a.e. in $\R^N$.
The uniqueness of $\mu$ follows arguing as in \cite[Theorem 8.1.15]{ber-lor}.
\end{proof}

The following proposition deals with the asymptotic behaviour of $(e^{-tA})_{t\geq0}$.
\begin{proposition}
\label{pro:asymp_behav}
For any $f\in L^2_\mu(\R^N)$ we have
\begin{align*}
(i) \ L^2{\rm -}& \lim_{t\rightarrow +\infty}e^{-tA}f  =\int_{\R^N}f\,d\mu, \\
(ii) \ L^2{\rm -} & \lim_{\lambda\rightarrow 0^+}\lambda R(\lambda,-A)f=\int_{\R^N}f\,d\mu.
\end{align*}
\end{proposition}
\begin{proof}
$(i)$  Let $f\in D( A)$. Then,
\begin{align*}
\frac{d}{dt}\int_{\R^N}|e^{-tA}f|^2d\mu
= & 2\int_{\R^N}\left(e^{-tA}f\right)\left( {-}Ae^{-tA}f\right) d\mu
= -2\int_{\R^N}|Le^{-tA}f|^2d\mu, \quad t>0.
\end{align*}
Hence,
\begin{align*}
2\int_0^t\|Le^{-sA}f\|_{L^2_\mu(\R^N)}^2ds
+\|e^{-tA}f\|_{L^2_\mu(\R^N)}^2\leq \|f\|_{L^2_\mu(\R^N)}^2, \quad t>0.
\end{align*}
This implies that the function $t\mapsto \Phi_f(t):=\|Le^{-tA}f\|_{L^2_\mu(\R^N)}^2\in L^1(0,+\infty)$. Further, if $f\in D(A^2)$ we get
\begin{align*}
\frac{d}{dt}\Phi_f(t)
= & 2\int_{\R^N}\left(L\frac{d}{dt}e^{-tA}f\right)\left(Le^{-tA}f\right)d\mu
=  2\int_{\R^N} \left(Le^{-tA}{(-A)}f\right)\left(Le^{-tA}f\right)d\mu
\leq 2\Phi_f(t)+2\Phi_{\textcolor{red}{-}Af}(t),
\end{align*}
for any $t>0$. Hence, both $\Phi_f$ and $\Phi'_f$ belong to $L^1(0,+\infty)$, which implies that
\begin{align*}
\lim_{t\rightarrow+\infty}\Phi_f(t)=0.
\end{align*}
Since $-A$ generates an analytic semigroup, we infer that for any $f\in L^2_\mu(\R^N)$ we have $e^{-A}f\in D( A^n)$ for any $n\in\N$. Therefore,
\begin{align*}
\lim_{t\rightarrow+\infty}\|Le^{-tA}f\|_{L^2_\mu(\R^N)}^2
= \lim_{t\rightarrow+\infty}\|Le^{-(t-1)A}e^{-A}f\|_{L^2_\mu(\R^N)}^2
= \lim_{t\rightarrow+\infty}\|Le^{-tA}e^{-A}f\|_{L^2_\mu(\R^N)}^2=0,
\end{align*}
for any $f\in L^2_\mu(\R^N)$.

Let us fix $f\in L^2_\mu(\R^N)$. Since $(e^{-tA})_{t\geq0}$ is a semigroup of contractions in $L^2_\mu(\R^N)$ it follows that $\sup_{t>0}\|e^{-tA}f\|_{L^2_\mu(\R^N)}<+\infty$, and so there exists a sequence $(t_n)$ diverging to $+\infty$ and $g\in L^2_\mu(\R^N)$ such that $e^{-t_nA}f\rightarrow g$ weakly in $L^2_\mu(\R^N)$. We claim that $g$ is constant. Indeed, for any $\psi\in D(L)$ we have
\begin{align*}
\int_{\R^N}gL\psi \,d\mu
= & \lim_{n\rightarrow+\infty}\int_{\R^N}(e^{-t_nA}f) L\psi \,d\mu
=  \lim_{n\rightarrow+\infty}\int_{\R^N}(Le^{-t_nA}f)\psi \,d\mu
=0.
\end{align*}
This means that $g\in D(L^*)=D(L)$ and $L^*g=Lg=0$. From \cite[Theorem 9.1.17]{lor-16} it follows that $g$ is a constant function and the claim is proved. Further,
\begin{align*}
g=\int_{\R^n}g\,d\mu
=\lim_{n\rightarrow+\infty}\int_{\R^n}e^{-t_nA}f\,d\mu
=\int_{\R^n}f\,d\mu,
\end{align*}
since $\mu$ is the unique invariant measure for $(e^{-tA})_{t\geq0}$. Above computations reveal that for any sequence $(t_m)$ diverging as $m\rightarrow+\infty$ there exists a subsequence $(t_{k_m})\subset (t_m)$ such that $e^{-t_{k_m}A}f\rightarrow \int_{\R^N}fd\mu$ weakly in $L^2_\mu(\R^N)$ as $m\rightarrow +\infty$. Hence, we get that $e^{-tA}f\rightarrow \int_{\R^N}fd\mu$ weakly in $L^2_\mu(\R^N)$ as $t\rightarrow+\infty$. To prove that $e^{-tA}f\rightarrow \int_{\R^N}fd\mu$ in $L^2_\mu(\R^N)$ as $t\rightarrow+\infty$, we notice that
\begin{align*}
\|e^{-tA}f\|_{L^2_\mu(\R^N)}^2
= & \int_{\R^N}(e^{-tA}f)(e^{-tA}f)d\mu
= \int_{\R^N}(e^{-2t A}f) fd\mu\rightarrow \int_{\R^N}gfd\mu=\|g\|_{L^2_\mu(\R^N)}^2,
\end{align*}
which gives the thesis.

\vspace{3mm}
$(ii)$  The statement is a consequence of $(i)$ and of the representation of $R(\lambda,-A)$ as Laplace transform of $(e^{-tA})_{t\geq0}$. Let us fix $f\in L^2_\mu(\R^N)$, then, 
\begin{align*}
R(\lambda,-A)f=\int_0^\infty e^{-\lambda t}e^{-tA}f dt, \quad \lambda>0.
\end{align*}
We have
\begin{align*}
\lambda R(\lambda,-A)f-\int_{\R^N}fd\mu
= & \lambda\int_0^\infty e^{-\lambda t}\left((e^{-tA}f-\int_{\R^N}fd\mu\right)dt.
\end{align*}
Let us fix $\varepsilon>0$, from $(i)$ there exists $\tau\in(0,+\infty)$ such that
\begin{align*}
\left\|e^{-tA}f-\int_{\R^N}fd\mu\right\|_{L_\mu^2(\R^N)}<\varepsilon/2, \quad t\geq \tau.
\end{align*}
By applying the integral Minkowski  inequality we get
\begin{align*}
& \left\|\lambda R(\lambda,-A)f-\int_{\R^N}fd\mu\right\|_{L^2_\mu(\R^N)} \\
& \qquad \leq  \lambda\int_0^\infty e^{-\lambda t}\left\|e^{-tA}f-\int_{\R^N}fd\mu\right\|_{L^2_\mu(\R^N)}dt \\
& \qquad =  \lambda\int_0^\tau e^{-\lambda t}\left\|e^{-tA}f-\int_{\R^N}fd\mu\right\|_{L^2_\mu(\R^N)}dt 
+ \lambda\int_\tau^{\infty} e^{-\lambda t}\left\|e^{-tA}f-\int_{\R^N}fd\mu\right\|_{L^2_\mu(\R^N)}dt \\
& \qquad <  2\|f\|_{L^2_\mu(\R^N)}\lambda\int_0^\tau e^{-\lambda t}dt+\frac\varepsilon 2 
 \leq  2\tau\lambda \|f\|_{L^2_\mu(\R^N)}+\frac\varepsilon 2.
\end{align*}
Let us choose $\overline \lambda>0$ such that $2\tau\lambda \|f\|_{L^2_\mu(\R^N)}\leq \varepsilon/2$ for any $\lambda\in(0,\overline \lambda)$, it follows that
\begin{align*}
 \left\|\lambda R(\lambda,-A)f-\int_{\R^N}fd\mu\right\|_{L^2_\mu(\R^N)}
< \varepsilon, \quad \lambda\in(0,\overline \lambda).
\end{align*}
The arbitrariness of $\varepsilon$ gives the thesis.
\end{proof}


\subsection{Local and asymptotic positivity}

In this subsection we study some positiveness properties of the semigroup $(e^{-tA})_{t\geq0}$. Let us begin with some considerations.
From \cite[Corollary 7.4(1)]{gre-mug} and Proposition \ref{pro:spectralA} it follows that the semigroup $e^{-tA}$ is individually asymptotically positive, i.e.,
\begin{align*}
f\in L^2_\mu(\R^N)_+ \Rightarrow \lim_{t\rightarrow+\infty}{\rm dist}(e^{-tA}f,L^2_\mu(\R^N)_+)=0,
\end{align*}
where $L^2_\mu(\R^N)_+:=\{f\in L^2_{\mu}(\R^N):f\geq0\}$. However, nothing can be said neither about uniform eventual positivity nor eventual irreducibility, in the sense of \cite{dan-glu-ken1,dan-glu-ken2}, since it is not known whether $D(A^k)\subset L^\infty(\R^N)$ for some $k\in\N$ or   $e^{t_0A}(L^2_\mu(\R^N))\subset L^\infty(\R^N)$ for some $t_0>0$. Proposition \ref{pro:asymp_behav} can be seen as an intermediate result between individually asymptotically positivity and eventually irreducibility, and we formulate it as follows. We introduce the space
\begin{align*}
L^2_\mu(\R^N)_>:=\{f\in L^2_\mu(\R^N):\ f\geq0, \ \exists A\in \mathcal B(\R^N): {\rm Leb}(A)>0, \ f(x)>0 \ x\in A\},
\end{align*}
where ${\rm Leb}(A)$ denotes the Lebesgue measure for any $A\subset \mathcal B(\R^N)$.
We say that a strongly continuous semigroup of linear bounded operators $T(t)$ on $L^2_\mu(\R^N)$ is asymptotically irreducible if
\begin{align*}
f\in L^2_\mu(\R^N)_>\Rightarrow \lim_{t\rightarrow+\infty}{\rm dist}(T(t)f,L^2_\mu(\R^N)_>)=0.
\end{align*}
\begin{proposition}
The semigroup $(e^{-tA})_{t\geq0}$ is asymptotically irreducible.
\end{proposition}
\begin{proof}
The statement follows combining Proposition \ref{pro:asymp_behav} and the fact that $\mu$ is equivalent to the Lebesgue measure. 
\end{proof}

The following results are inherit from the abstract results in \cite[Theorems 3.3 \& 4.4]{aro} applied to our setting.   In particular, in \cite{aro}, criteria for individual and uniform local eventual positivity for $C_0$-semigroups are proved. More precisely, given three Banach lattices $E,F,G$,   a $C_0$-semigroup $\left(e^{tA}\right)_{t\geq0}$ with generator $A:D(A)\subset F\to F$ and two positive operators $S\in\mathcal{L}(F,G), T\in\mathcal{L}(E,F)$, positivity of $Se^{tA}Tf$ for large times $t>0$ whenever $f$ is positive is stated.  
We provide explicit proofs which follow the lines of that of the quoted theorems, but are simplified since we deal with concrete objects.


\begin{proposition}
\label{prop:loc_pos_1}
Assume that there exists $n\in\N$ such that $D(A^n)\subset L^\infty_{\rm loc}(\R^N)$. Then, the semigroup $(e^{-tA})_{t\geq0}$ is locally individually eventually positive, i.e., for any $f\in L^2_\mu(\R^N)\cap L^2_\mu(K)_>$ and any $K\subset \R^N$ compact set, there exist $c>0$ and $t_0>0$ such that
\begin{align}
\label{local_ind_positivity}
e^{-tA}(\chi_Kf)(x)\geq c, \quad t\geq t_0, \quad \textup{a.e. in }x\in K. 
\end{align}
\end{proposition}
\begin{proof}
Let $K\subset \R^N$ be a compact set. Let us prove that we can apply to $(e^{-tA})_{t\geq0}$ the procedure in \cite[Theorem 3.3]{aro}, with $E=F=G=L^2_\mu(\R^N)$, and $S f=Tf=\chi_Kf$ for any $f\in L^2_\mu(\R^N)$. Clearly, $T'f=\chi_kf$ for any $f\in L^2_\mu(\R^N)$.

We recall that
$G_{\chi_K}:=\left\{f\in G|\ \!\exists c>0\ \! : \ \! |f| \leq c\chi_K\right\}$ is a Banach space if endowed with the norm $\|f\|_{G_{\chi_K}}:=\inf\{c\geq 0 :\ |f|\leq c\chi_K\}$ for any $f\in G_{\chi_K}$. Let us notice that $f\in G_{\chi_K}$ if and only if $f$ is essentially bounded on $K$ and $f\equiv0$ a.e. on $\R^N\smallsetminus K$. The analyticity of $A$ implies that  $e^{-tA}f\subset D(A^k)$ for any $f\in L^2_\mu(\R^N)$, any $k\in\N$ and any $t>0$, hence for any $f\in L^2_\mu(\R^N)$ we have $Se^{-tA}f\in G_{\chi_K}$ for any $t>0$. 

Arguing as in \cite[Theorem 3.3]{aro} we infer that $Se^{tA}Tf\rightarrow SP_\infty Tf$ in $G_{\chi_K}$ as $t\rightarrow+\infty$ for any $f\in L^2_\mu(\R^N)$, i.e., 
\begin{align*}
\lim_{t\rightarrow+\infty}\inf\{c>0:|Se^{tA}Tf- SP_\infty Tf|\leq c\chi_K \textup{ a.e. in }K\}=0, \quad f\in L^2_\mu(\R^N).
\end{align*}
Let us consider $f\in L^2_\mu(\R^N)\cap L_\mu^2(K)_>$ and let us notice that
\begin{align*}
SP_\infty Tf=\chi_K\|f\|_{L_\mu^1(K)}>0.
\end{align*}
If $t_0>0$ satisfies $\inf\{c>0:|Se^{tA}Tf- SP_\infty Tf|\leq c\chi_K \textup{ a.e. in }K\}\leq 2^{-1}\|f\|_{L^1_\mu(\R^N)}$  for any $t\geq t_0$, by combining the above computations we get
\begin{align*}
Se^{-tA}Tf=Se^{tA}Tf- SP_\infty Tf+SP_\infty T f\geq \frac{\|f\|_{L^1_\mu(\R^N)}}2\chi_K, \quad \textup{a.e. in }K, \textup{ for any }t\geq t_0.
\end{align*}
From the definition of $S$ and $T$ \eqref{local_ind_positivity} follows.
\end{proof}

If $0$ is a simple pole for $\sigma(-A)$ then we can improve the result of Proposition \ref{prop:loc_pos_1}.

\begin{proposition}
\label{prop:loc_pos_2}
Assume that there exists $n\in\N$ such that $D(A^n)\subset L^\infty_{\rm loc}(\R^N)$ and that $0$ is a simple pole for $\sigma(-A)$. Then, the semigroup $(e^{-tA})_{t\geq0}$ is locally uniformly eventually positive, i.e., for any $K\subset \R^N$ compact set there exists $t_0>0$ such that for any $f\in L^2_\mu(\R^N) \cap L^2_\mu(K)_>$ there exists $c>0$ which satisfies
\begin{align*}
e^{-tA}(\chi_Kf)(x)\geq c, \quad  t\geq t_0, \ \textup{a.e. in }x\in K.
\end{align*}
\end{proposition}
\begin{proof}
Since $0$ is a simple pole for $\sigma(-A)$ and $(e^{-tA})_{t\geq0}$ is eventually norm continuous (this fact follows from the analyticity of $(e^{-tA})_{t\geq0}$, se e.g. \cite[Chapter II, Section 4.c]{EnNa00}), from \cite[Theorem 2.7]{Th98} we have $e^{-tA}\rightarrow P_\infty$ in operator norm as $t\rightarrow+\infty$.

Let us fix a compact set $K$, let $E:=L^2_\mu(K)$ and let $F,G,S,T$ be as in the proof of Proposition \ref{prop:loc_pos_1}. Arguing as \cite[Theorem 4.4]{aro} we infer that $Se^{-tA}(I-P_\infty)T\rightarrow0$ in $\mathcal L(E^{T'},G_{\chi_K})$, where $E^{T'}$ is the closure of $E$ with respect to the norm 
\begin{align*}
\|f\|_{E^{T'}}:=\int_K|f|d\mu, \quad f\in L^2_\mu(K). 
\end{align*}
Hence, for any $\varepsilon>0$  there exists $t_0>0$ such that
\begin{align*}
\|Se^{-tA}(I-P_\infty)Tf\|_{G_{\chi_K}}\leq \varepsilon \|f\|_{E^{T'}}=\varepsilon\|f\|_{L^1_\mu(K)}, \quad f\in L^2_\mu(K).
\end{align*}
Let us choose $f\in L^2_\mu(K)_>$, then $SP_\infty Tf=\chi_K\|f\|_{L^1_\mu(K)}$ and we have
\begin{align*}
Se^{-tA}Tf(x)
= & SP_\infty Tf(x)+ Se^{-tA}(I-P_\infty)Tf(x)
\geq (1-\varepsilon)\|f\|_{L^1_\mu(K)}\chi_K(x), \quad \textup{a.e. in }x\in K, \ t\geq t_0.
\end{align*}
By choosing $\varepsilon=1/2$ we get the thesis with $c=2^{-1}\|f\|_{L^1_\mu(K)}$.
\end{proof}

\begin{remark}
\label{rmk:cond_loca_pos} In order to apply the above results, it remains then to provide sufficient conditions which ensure $D(A^n)\subset L^\infty_{\rm loc}(\R^N)$ for some $n\in\N$ and that $0$ is a simple pole for $\sigma(-A)$. 

\noindent To get the first condition, from the Sobolev embeddings, it is enough to prove that $D(A^n)\subset H^{j}_{\rm loc}(\R^N)$ for some $j>N/2$. Let us notice that if $\mu\in C^3(\R^N)$, by applying the local elliptic regularity and the explicit formula \eqref{explicit_formula_A} of $A$ it follows that $D(A)\subset H^3_{\rm loc}(\R^N)$. Analogously, we can prove that if $\mu\in C^{4n-1}(\R^N)$ it follows that $D(A^n)\subset H_{\rm loc}^{4n-1}(\R^N)$ for any $n\in\N$. Hence, if $n\in\N$ satisfies $4n-1>N/2$ and $\mu\in C^{4n-1}(\R^N)$, it follows that $D(A^n)\subset H^{4n-1}_{\rm loc}(\R^N)\subset L^\infty_{\rm loc}(\R^N)$.

\noindent As far as the second request is concerned, let us notice that if the embedding $D(L)\subset L^2_\mu(\R^N)$ is compact, then the operator $-A$ has compact resolvent. Further, applying \cite[Chapter IV, Corollary 1.19]{EnNa00} it follows that $0$ is a pole for $\sigma(-A)$ and Proposition \ref{pro:spectralA} ensures that the multiplicity of ${\rm rg P_\infty}=1$, i.e., $0$ is a simple pole for $\sigma(-A)$. By Theorem 3.1 and the successive Example in \cite{Ho81},  it is enough that $\mu$ decays at infinity as $e^{-c|x|^\alpha}$ for some $\alpha>1$ and $c>0$ to obtain that the embedding $H^k_\mu(\R^N)\subset L^2_\mu(\R^N)$ is compact for any $k\in\N$, $k\geq1$.  To conclude, since $D(L)\subset H^1_\mu(\R^N)$, a sufficient condition for $0$ of being a simple pole for $\sigma(-A)$ is that the embedding $H^1_\mu(\R^N)\subset L^2_\mu(\R^N)$ is compact and therefore that $\mu$ decays at infinity as $e^{-c|x|^\alpha}$ for some $\alpha>1$ and $c>0$.
\end{remark}

Taking Remark \ref{rmk:cond_loca_pos}  into account, we provide an example of a measure $\mu$ which satisfies the assumptions of Proposition \ref{prop:loc_pos_2}.

\begin{example}\label{ex:measure}
The measure 
\begin{align*}
\mu(x):=K\exp\left(-(c_1+c_2|x|^2)^m\right), x\in \R^N,
\end{align*}
where $c_1,c_2,m$ are positive constants with $m>1/2$ and $K:=\|\mu\|_{L^1(\R^N)}^{-1}$ is a normalizing factor, satisfies the assumptions of Proposition \ref{prop:loc_pos_2}.
\end{example}

\section{Weighted Rellich's inequality}\label{Rel-ineq} 

The aim of this section is to prove a weighted Rellich's inequality with respect to the measure $\mu$. We define
\[U:=\frac14\left|\frac{\nabla\mu}{\mu}\right|^2-\frac12\frac{\Delta\mu}{\mu},\]
denote by $C_0:=\left(\frac{N-2}{2}\right)^2$
and consider the following assumptions on $\mu$.\\

\textbf{Hypothesis} $(H2)$ 
\begin{itemize}
\item[(i)] $\sqrt{\mu}\in H^1_{\rm loc}(\R^N),\Delta\mu\in L^1_{\rm loc}(\R^N)$;
\item[(ii)] there exists a $R_0>0$ such that $|x|^2U(x)\leq\frac14\frac{1}{|\log|x||^2}\quad\forall\,x\in B_{R_0}$;
\item[(iii)] $U$ is bounded from above in $\R^N\setminus B_R$ for every $R>0$.
\end{itemize}

We note that if $\mu$ satisfies $(H1)(i)$ then $\sqrt{\mu}\in H^1_{\rm loc}(\R^N)$. In fact,  ${\mu}\in H^1_{\rm loc}(\R^N)$ implies $\mu\in L^1_{\rm loc}(\R^N)$ and $\nabla\mu\in L^2_{\rm loc}(\R^N,\R^N)$ and since $\frac{\nabla \mu}{\mu}\in L^r_{\rm loc}(\R^N,\R^N)$ for some $r>N$, it follows that $\frac{\nabla \mu}{\mu}\in L^2_{\rm loc}(\R^N,\R^N)$. Then $\sqrt\mu\in L^2_{\rm loc}(\R^N)$ and 
$$\int_K|\nabla \mu^{\frac{1}{2}}|^2dx=\frac{1}{4}\int_K\frac{|\nabla \mu|^2}{\mu}dx
\le \frac{1}{4}\left(\int_K\left|\frac{\nabla \mu}{\mu}\right|^2dx\right)^{\frac{1}{2}}
\left(\int_K|\nabla \mu|^2dx\right)^{\frac{1}{2}}<\infty$$
for every compact set $K\subset \R^N$.

\begin{theorem}
Assume $N\geq5$ and Hypothesis $(H2)$.  Then,
\begin{enumerate}[(i)]
\item For any $\varepsilon>0$ and any $u\in C_c^\infty(\R^N)$ we have
\begin{align}\label{weighted-Rellich0}
\left((C_0-1)^2-\varepsilon\right)\int_{\R^N}\frac{|u(x) |^2}{|x|^4}d\mu \leq \int_{\R^N}|Lu(x)|^2d\mu +\frac{\left((C_0-1)C_1\right)^2}{\varepsilon}\int_{\R^N}|u(x) |^2d\mu .
\end{align}
\item For any $u\in C_c^\infty(\R^N)$ we have
\begin{align}\label{weighted_Rellich}
(C_0-1)^2\int_{\R^N}\frac{|u(x) |^2}{|x|^4}d\mu 
\leq & \int_{\R^N}|Lu(x) |^2d\mu+\frac{2(C_0-1)C_1}{C_0}\int_{\R^N}|\nabla u(x)|^2d\mu  \\\nonumber& 
+\frac{2(C_0-1)(C_1)^2}{C_0}\int_{\R^N}|u(x) |^2d\mu .
\end{align}
\end{enumerate} 
\end{theorem}
\begin{proof}
Under the assumption (H2), by \cite[Theorem 3.3]{can-gre-rha-tac}, the following weighted Hardy inequality holds 
\begin{align}
\label{weigthed_Hardy}
C_0\int_{\R^N}\frac{|\varphi |^2}{|x|^2}d\mu
\leq \int_{\R^N}|\nabla \varphi |^2d\mu+C_1\int_{\R^N}|\varphi |^2d\mu, \quad \varphi\in 
H^1_\mu(\R^N) 
\end{align}
for some constant $C_1>0$ and $C_0=\left(\frac{N-2}{2}\right)^2$ is the best constant in Hardy's inequality.  For $u\in C_c^\infty(\R^N)$ let
us apply \eqref{weigthed_Hardy} to the function $\varphi(x)=u(x)(|x|^2+\delta)^{-1/2}$ for $\delta>0$. We have
\begin{align}
\label{weighted_hardy_1}
C_0\int_{\R^N}\frac{|u |^2}{|x|^2(|x|^2+\delta)}d\mu 
\leq & \int_{\R^N}\left|\nabla\left(\frac{u }{(|x|^2+\delta)^{1/2}}\right)\right|^2d\mu +C_1\int_{\R^N}\frac{|u |^2}{|x|^2+\delta}d\mu.
\end{align}
Since
\begin{align*}
\left|\nabla\left(\frac{u }{(|x|^2+\delta)^{1/2}}\right)\right|^2
= \left|\frac{\nabla u }{(|x|^2+\delta)^{1/2}}-u \frac{x}{(|x|^2+\delta)^{3/2}}\right|^2 
=  \frac{|\nabla u |^2}{|x|^2+\delta}-2u \frac{  \nabla u\cdot x }{(|x|^2+\delta)^2}+\frac{|u |^2|x|^2}{(|x|^2+\delta)^3},
\end{align*}
by applying the integration by parts formula to the above first addend we infer that
\begin{align}\label{conto_grad_weighted_hardy}
\int_{\R^N}\left|\nabla\left(\frac{u }{(|x|^2+\delta)^{1/2}}\right)\right|^2d\mu 
= & \int_{\R^N}\left(\frac{  \nabla u\cdot \nabla u  }{|x|^2+\delta}-2u \frac{  \nabla u \cdot x }{(|x|^2+\delta)^2}+\frac{|u |^2|x|^2}{(|x|^2+\delta)^3}\right)d\mu  \nonumber \\
= & \int_{\R^N} \nabla u \cdot\nabla\left(\frac{u }{|x|^2+\delta}\right)  d\mu + \int_{\R^N}\frac{|u|^2|x|^2}{(|x|^2+\delta)^3}d\mu  \nonumber \\
= & -\int_{\R^N}Lu \frac{u }{|x|^2+\delta}d\mu +\int_{\R^N}\frac{|u |^2|x|^2}{(|x|^2+\delta)^3} d\mu .
\end{align}
By replacing \eqref{conto_grad_weighted_hardy} in \eqref{weighted_hardy_1} and using the fact that $|x|^2<|x|^2+\delta$, we get
\begin{align*}
(C_0-1)\int_{\R^N}\frac{|u |^2}{(|x|^2+\delta)^2}d\mu 
\leq -\int_{\R^N}Lu \frac{u}{|x|^2+\delta}d\mu 
+C_1\int_{\R^N}\frac{|u |^2}{|x|^2+\delta}d\mu.
\end{align*}
Let us apply the Young inequality $ab\leq \frac{\eta}{2} a^2+\frac{1}{2\eta}b^2$, for any $a,b,\eta>0$, to the first addend of the right-hand side above, with $a=|u(x) |(|x|^2+\delta)^{-1}$ and $b=|Lu(x) |$. It follows that
\begin{align*}
-\int_{\R^N}Lu \frac{u }{|x|^2+\delta}d\mu 
\leq \frac{1}{2\eta}\int_{\R^N}|Lu |^2d\mu +\frac{\eta}{2}\int_{\R^N}\frac{|u |^2}{(|x|^2+\delta)^2}d\mu ,
\end{align*}
which gives
\begin{align*}
2\eta\left(C_2-\frac{\eta}{2}\right)\int_{\R^N}\frac{|u |^2}{(|x|^2+\delta)^2}d\mu 
\leq \int_{\R^N}|Lu|^2d\mu +2C_1\eta\int_{\R^N}\frac{|u|^2}{|x|^2+\delta}d\mu ,
\end{align*}
where $C_2:=C_0-1$. The maximum of the function $(0,+\infty)\ni \eta\mapsto 2\eta\left(C_2-\frac{\eta}{2}\right)$ is achieved at $\eta=C_2$. Then,
\begin{align}
\label{weighted_Hardy_2}
(C_2)^2\int_{\R^N}\frac{|u |^2}{(|x|^2+\delta)^2}d\mu\leq \int_{\R^N}|Lu |^2d\mu +2C_1C_2\int_{\R^N}\frac{|u |^2}{|x|^2+\delta}d\mu .
\end{align}
To prove $(i)$, it is enough to apply again the Young inequality $2ab\leq \widetilde{\delta} a^2+\frac{1}{\widetilde{\delta}}b^2$ for any $a,b,\widetilde{\delta}>0$ to the last addend in the right-hand side of \eqref{weighted_Hardy_2} with $a=|u(x)|(|x|^2+\delta)^{-1}$ and $b=|u(x)|$. It follows that
\begin{align*}
(C_2)^2\int_{\R^N}\frac{|u |^2}{(|x|^2+\delta)^2}d\mu \leq \int_{\R^N}|Lu |^2d\mu +\widetilde{\delta}{C_1C_2}\int_{\R^N}\frac{|u |^2}{(|x|^2+\delta)^2}d\mu+\frac{C_1C_2}{\widetilde{\delta}}\int_{\R^N}{|u |^2}d\mu,
\end{align*}
for any $\widetilde{\delta}>0$. For $\varepsilon>0$ take $\widetilde{\delta}=\varepsilon(C_1 C_2)^{-1}$. Then we have
\begin{align*}
\left((C_2)^2-\varepsilon\right)\int_{\R^N}\frac{|u |^2}{(|x|^2+\delta)^2}d\mu \leq \int_{\R^N}|Lu |^2d\mu+\frac{(C_1C_2)^2}{\varepsilon}\int_{\R^N}{|u |^2}d\mu.
\end{align*}
Letting $\delta\to 0$, the thesis follows by applying Fatou's lemma.
\\
For assertion $(ii)$ we use \eqref{weighted_Hardy_2} to deduce that
\begin{equation}\label{eq:4-6}
(C_2)^2\int_{\R^N}\frac{|u |^2}{(|x|^2+\delta)^2}d\mu \le 
\int_{\R^N}|Lu |^2d\mu +2C_1C_2\int_{\R^N}\frac{|u|^2}{|x|^2}d\mu.
\end{equation}
Now, applying \eqref{weigthed_Hardy} to the last integral in the right-hand side of \eqref{eq:4-6} we obtain
\begin{align*}
(C_0-1)^2\int_{\R^N}\frac{|u |^2}{(|x|^2+\delta)^2}d\mu
\leq & \int_{\R^N}|Lu |^2d\mu +\frac{2(C_0-1)C_1}{C_0}\int_{\R^N}|\nabla u |^2d\mu  \\& 
+\frac{2(C_0-1)(C_1)^2}{C_0}\int_{\R^N}|u |^2d\mu.
\end{align*}
So, Fatou's lemma gives the assertion $(ii)$.
\end{proof}

\begin{remark}\label{rem:ul-mom}
If one assumes $(H1)$ and $(H2)$, then, since $C_c^\infty(\R^N)$ is a core for $L$, it follows that \eqref{weighted-Rellich0} and \eqref{weighted_Rellich} holds true for all $u\in D(L)$.
\end{remark}

\subsection{Optimality of the constant}
In this subsection we prove that the best constant in the weighted Rellich inequality is $(C_0-1)^2=\left( \frac{N(N-4)}{4} \right)^2$, which is known as the best constant in the classical Rellich inequality.

Using Remark \ref{rem:2-1}, we first observe the following properties of the measure $\mu$.
\begin{remark}\label{regmu}
Assume that $\mu$ satisfies $(H1)$. Then
\begin{itemize}
\item[(i)]  $\sup_{\delta\in \R}\left\{ \frac{1}{|x|^{\delta}}\in L^{1}_{\mu,\rm loc}(\R^{N}) \right\}=N$. Indeed, from Remark \ref{rem:2-1}, for any $\varepsilon>0$ we have
\begin{align*}
\int_{B_1\setminus B_\varepsilon}|x|^{-\delta}d\mu
\sim \int_{B_1\setminus B_\varepsilon}|x|^{-\delta}dx
= \omega_N\int_{\varepsilon}^1t^{-\delta+N-1}dt
=\begin{cases}
\displaystyle \omega_N\frac{(1-\varepsilon^{N-\delta})}{N-\delta}, & \delta<N, \\
-\omega_N\ln(\varepsilon) & \delta=N;
\end{cases}
\end{align*}
\item[(ii)] for any $\delta<N$ and any $n\in\N$ we have
\begin{align*}
\int_{B_{1/n}}|x|^{-\delta} d\mu =\omega_N\left(\frac{1}{n}\right)^{-\delta+N}.
\end{align*}
This simply follows arguing as in $(i)$.
\end{itemize}
\end{remark}

 Taking into consideration Remark \ref{rem:ul-mom} and  Remark \ref{regmu} we prove now the optimality of the constant $(C_0-1)^2$ in the weighted Rellich  inequality \eqref{weighted_Rellich}.

\begin{theorem}\label{Th: Optimal}
Assume that $N\geq5$, Hypothesis $(H1)$ holds and for any $\varepsilon >0$ there is $C_\varepsilon >0$ such that 
$$\left|\frac{\nabla \mu}{\mu}\right|^2\le \frac{\varepsilon}{|x|^2}+C_\varepsilon \quad \hbox{\ in } B_{R_1}$$
for some $R_1>0$.
Then, the weighted Rellich inequality 
\begin{align*}
c\int_{\R^N}\frac{|u(x)|^2}{|x|^4}d\mu
\leq \int_{\R^N}|L u(x)|^2d\mu+C\|u\|^{2}_{H^{1}_{\mu}}\,\text{ for some }C>0\text{ and all }u\in D(L)
\end{align*}
does not hold if $c>(C_0-1)^2=\left( \frac{N(N-4)}{4} \right)^2$.
\end{theorem}
{\sc Proof.}
Take $\gamma<0$ and consider the function $\varphi(x)=|x|^\gamma$. A  simple computation shows that
\[
L\varphi=\Delta |x|^{\gamma}+\frac{\nabla \mu}{\mu}\cdot \nabla |x|^{\gamma}=
\gamma \left( \gamma-2+N+ \frac{\nabla \mu}{\mu} \cdot x\right)|x|^{\gamma-2}.
\]
Using Young's inequality and the hypotheses on $\mu$, for $x\in B_{R_1}$, we have
\begin{align}
&|L\varphi |^{2}\leq
	(1+\varepsilon)\gamma^2(\gamma-2+N)^2|x|^{2\gamma-4}
	+\left( 1+\frac{1}{\varepsilon} \right)\gamma^{2}\left |\frac{\nabla \mu}{\mu} \right |^{2}|x|^{2\gamma-2}\nonumber \\
&\quad \leq \left[
		(1+\varepsilon)\gamma^2(\gamma-2+N)^2
			+\varepsilon^{2}\left( 1+\frac{1}{\varepsilon} \right)\gamma^{2}
		 \right]|x|^{2\gamma-4}
		 +C_{\varepsilon^{2}}\left( 1+\frac{1}{\varepsilon} \right)\gamma^{2}|x|^{2\gamma-2}\nonumber \\
&\quad \leq \left[
		(1+\varepsilon)\gamma^2(\gamma-2+N)^2
			+(\varepsilon^{2}+\varepsilon)\gamma^{2}
		 \right]|x|^{2\gamma-4}
		 +C|x|^{2\gamma-2}.\label{eq:modifie}	 
\end{align}
%
%

Let $c>(C_0-1)^2$ and $\gamma$ be such that 
\begin{equation*} 
\begin{cases}
(1+\varepsilon)\gamma^2(\gamma-2+N)^2+(\varepsilon^{2}+\varepsilon)\gamma^{2} < c\\
2\gamma-4\leq -N\\
2\gamma-2>-N
\end{cases}
\end{equation*}
so that $\varphi\in H^1_{\rm loc}(\R^N)$ but $\varphi\notin H^2_{\rm loc}(\R^N)$. Observe that such choice of $\gamma $ is possible for $\varepsilon$ small enough since the function $f(\gamma)=\gamma^2\left( \gamma-2+N \right)^2$ is  continuous  and decreasing in
$\left(1-\frac{N}{2},2-\frac{N}{2}\right]$   and $f(\gamma)\geq f(2-\frac{N}{2})=(C_0-1)^2$.

For $n\in\N$ and $\vartheta\in C_c^\infty(\R^N)$ such that $0\leq \vartheta\leq 1$,
$\vartheta=1$ in $B_1$ and $\vartheta=0$ in $B_2^c$ we
set 
\[
\varphi_n(x)=\left\{
\begin{array}{cl}
\alpha_n+\beta_n|x|^{\gamma_1} &\text{ if }|x|<\frac{1}{n},\\
|x|^\gamma&\text{ if }\frac{1}{n}\leq |x|< 1,\\
|x|^\gamma\vartheta(x) &\text{ if }1\leq |x|<2,\\
0&\text{ if }|x|\geq 2,\\
\end{array}
\right.
\]
where $2-\frac{N}{2}<\gamma_1<0$, $\beta_n=\frac{\gamma_1}{\gamma}\frac{1}{n^{\gamma-\gamma_1}}$ and 
$\alpha_n=\frac{1}{n^\gamma}-\frac{\beta_n}{n^{\gamma_1}}$. 
Since $L$ is the operator associated to the form $a$ given by \eqref{form-L}, on can see that $\varphi_n\in D(L)$. Modifying the support of $\varphi_n$ we can assume without loss of generality that $R_1=1$.

We propose now to prove that 
\[
\lambda_1:=\inf_{\varphi \in H^2_\mu\setminus \{0\} }
\left(
  \frac{\int_{\R^N} \left( |L \varphi|^2-\frac{c}{|x|^4}
\varphi^2 \right)\, d\mu}{\int_{\R^N} \varphi ^{2}\,d\mu+\int_{\R^N} |\nabla \varphi |^{2}\,d\mu}
\right)
\]
is equal to $-\infty$.
To this purpose we observe that  
\begin{align*}
&\int_{\R^N} \left( |L \varphi_n|^2-\frac{c}{|x|^4}\varphi_n^2 \right)\, d\mu
	 =\int_{B_1\setminus B_{\frac{1}{n}}} \left( \left| L |x|^\gamma \right|^2-
	\frac{c}{|x|^4}|x|^{2\gamma} \right)\, d\mu\nonumber\\
&\qquad
	+\int_{B_2\setminus B_1} \left( |L\left( |x|^\gamma \vartheta(x) \right)|^2-\frac{c}{|x|^4}
	\left( |x|^\gamma \vartheta(x) \right)^2 \right)\, d\mu\nonumber\\
&\qquad  
	+\int_{B_{\frac{1}{n}}}   |L \left(\alpha_n+\beta_n|x|^{\gamma_1}  \right)|^{2}
	-\frac{c}{|x|^{4}} \left(\alpha_n+\beta_n|x|^{\gamma_1}  \right)^2   d\mu.
\end{align*}
So, by \eqref{eq:modifie}, we have
\begin{align*}
&\int_{\R^N} \left( |L \varphi_n|^2-\frac{c}{|x|^4}\varphi_n^2 \right)\, d\mu
 	\leq \left[(1+\varepsilon)\gamma^2(\gamma-2+N)^2
	+(\varepsilon^{2}+\varepsilon)\gamma^{2}
	-c\right]\int_{B_1\setminus B_{\frac{1}{n}}} |x|^{2\gamma-4}\, d \mu\\
&\quad+ C\int_{B_1\setminus B_{\frac{1}{n}}}|x|^{2\gamma-2}
	+\int_{B_2\setminus B_1} \left( |L\left( |x|^\gamma \vartheta(x) \right)|^2-\frac{c}{|x|^4}
	\left( |x|^\gamma \vartheta(x) \right)^2 \right)\, d\mu \\
&\quad + \int_{B_{\frac{1}{n}}}   \beta_n^2
	\left[
		(1+\varepsilon)\gamma_{1}^2(\gamma_{1}-2+N)^2
			+(\varepsilon^{2}+\varepsilon)\gamma_1^{2}
		 \right]|x|^{2\gamma_{1}-4}
        -c\left(\alpha_n+\beta_n|x|^{\gamma_1}  \right)^2\frac{1}{|x|^4}\,d\mu \\
&\quad 	+ C\int_{B_{\frac{1}{n}}}  |x|^{2\gamma_{1}-2}d\mu.
\end{align*}
The term $\int_{B_{\frac{1}{n}}}   \beta_n^2
	\left[
		(1+\varepsilon)\gamma_{1}^2(\gamma_{1}-2+N)^2
			+(\varepsilon^{2}+\varepsilon)\gamma_1^{2}
		 \right]|x|^{2\gamma_{1}-4}
        -c\left(\alpha_n+\beta_n|x|^{\gamma_1}  \right)^2\frac{1}{|x|^4}\,d\mu $ is negative, for $\varepsilon$ small enough, since $f(\gamma_1)\leq c$, $\alpha_n>0$ and $\beta_n>0$. Moreover, using $(H1)$ we deduce that 
 $$\int_{B_2\setminus B_1} \left( |L\left( |x|^\gamma \vartheta(x) \right)|^2-\frac{c}{|x|^4}
	\left( |x|^\gamma \vartheta(x) \right)^2 \right)\, d\mu \le C_\vartheta$$
	for some constant $C_\vartheta >0$.
Furthermore, since $2\gamma_{1}-2>2\gamma-2>-N$, we obtain
$$\int_{B_{\frac{1}{n}}}|x|^{2\gamma_1-2}+\int_{B_1\setminus B_{\frac{1}{n}}}|x|^{2\gamma-2}d\mu \le C$$
for some positive constant $C$.
Then we have
\begin{equation}\label{eq:stima0}
\int_{\R^N} \left( |L \varphi_n|^2-\frac{c}{|x|^4}\varphi_n^2 \right)\, d\mu
 	\leq \left[(1+\varepsilon)\gamma^2(\gamma-2+N)^2
	+(\varepsilon^{2}+\varepsilon)\gamma^{2}
	-c\right]\int_{B_1\setminus B_{\frac{1}{n}}} |x|^{2\gamma-4}\, d \mu+C_{\vartheta}+C.
\end{equation}
%
%
On the other hand we know that $\|\varphi_n\|_{H^1_{\mu}}\geq C'_\vartheta$ for some $C'_\vartheta >0$.
Putting together this and \eqref{eq:stima0} one obtains
\[
\lambda_1
\leq \frac{ \left[(1+\varepsilon)\gamma^2(\gamma-2+N)^2+(\varepsilon^{2}+\varepsilon)\gamma^{2}-c\right]\int_{B_1\setminus B_{\frac{1}{n}}} |x|^{2\gamma-4}\, d\mu+C_\vartheta +C}{C'_\vartheta }.
\]
%
Now, since $2\gamma-4\leq-N$ and
taking into account Remark \ref{regmu} (i) we have
\[
 \lim_{n\to \infty }\int_{B_1\setminus B_{\frac{1}{n}}}|x|^{2\gamma-4}\, d\mu=+\infty 
\]
Therefore, since $(1+\varepsilon)\gamma^2(\gamma-2+N)^2+(\varepsilon^{2}+\varepsilon)\gamma^{2}-c<0$ it follows that
$\lambda_1=-\infty$.
\qed

\section{Domain characterization}\label{domain-bi-kolm}

In this section we want to characterise the domain of $A$. To this purpose we first have to characterise the domain of $L$.
Then, recall that  $U=-\frac{ \Delta \sqrt{\mu}}{\sqrt 	\mu}=-\frac{1}{2}\frac{\Delta \mu}{\mu}+\frac{1}{4}\frac{|\nabla \mu|^{2}}{\mu^{2}}$, we first prove the following fundamental estimate.
\begin{proposition}\label{pr:prop-stima-drift}
Let $N\geq3$. Assume that $\mu$ satisfies Hypothesis $(H2)$.
If $\varphi\in H^{1}_{\mu}(\R^{N})$
then $\frac{\nabla \mu}{\mu}\varphi\in L^{2}_{\mu}(\R^N,\R^N)$ and there exists $C>0$ such that
$$
\left \|\frac{\nabla \mu}{\mu}\varphi\right \|_{L^{2}_\mu} \leq C\left (\|\nabla \varphi\|_{L^{2}_\mu}+\|\varphi\|_{L^2_\mu}\right).
$$
\end{proposition}
{\sc Proof.}
Integrating by parts, taking into account the definition of $U$ and using H\"older's and Young's inequality we have
\begin{align*}
&\left \|\frac{\nabla \mu}{\mu}\varphi\right \|^{2}_{L^{2}_\mu}
	=\int_{\R^N} \nabla \mu \frac{\nabla \mu}{\mu}\varphi^{2}dx
	=-\int_{\R^N} \left (\frac{\Delta \mu}{\mu}\varphi^{2}-\frac{|\nabla \mu|^{2}}{\mu^{2}}\varphi^{2}
		+2\varphi \frac{\nabla \mu}{\mu}\cdot \nabla\varphi \right )d\mu\\
&\quad \leq	\int_{\R^N} 2U\varphi^{2}d\mu+\frac{1}{2}\int_{\R^N} \frac{|\nabla \mu|^{2}}{\mu^{2}}\varphi^{2}d\mu+
		\delta \int_{\R^N} \frac{|\nabla \mu|^{2}}{\mu^{2}}\varphi^{2}d\mu
			+\frac{C}{\delta }\int_{\R^N} |\nabla \varphi |^{2}d\mu
\end{align*}
for every $\delta>0$ and some $C>0$. Then 
\begin{equation}\label{eq:pr-stima-drift1}
\left(\frac{1}{2}-\delta\right)\left \|\frac{\nabla \mu}{\mu}\varphi\right \|^{2}_{L^{2}_\mu}\leq 
	\int_{\R^N} 2U\varphi^{2}d\mu+\frac{C}{\delta}\|\nabla \varphi\|^{2}_{L^{2}_\mu}.
\end{equation}
We observe now that the hypotheses of \cite[Theorem 3.3]{can-gre-rha-tac} are fulfilled
and then the weighted Hardy inequality 
\[
c\int_{\R^N} \frac{\varphi^{2}}{|x|^{2}} d\mu
\leq \int_{\R^N} |\nabla \varphi|^{2}d\mu+c_{\mu}\|\varphi\|_{L^2_\mu}^{2}
\]
holds for $c<C_{0} =\left (\frac{N-2}{2}\right)^{2}$, where $c_{\mu}$ is a positive constant. 
Moreover, by $(ii)$ and $(iii)$ of $(H2)$,   there exist  
$c_{1}< C_{0}$ and $c_{2} \geq 0$ such that  
$2U\leq \frac{c_{1}}{|x|^{2}}+c_{2}$. Then we have
\begin{equation}\label{eq:pr-stima-drift2}
\int_{\R^N} 2U\varphi^{2}d\mu\leq c_{1}\int_{\R^N} \frac{\varphi^{2}}{|x|^{2}} d\mu+c_{2}\int_{\R^N} \varphi^{2}d\mu
\leq \|\nabla \varphi\|_{L^{2}_\mu}^{2}+(c_{2}+c_{\mu})\|\varphi\|_{L^{2}_\mu}^{2}.
\end{equation}
Putting together \eqref{eq:pr-stima-drift1}, \eqref{eq:pr-stima-drift2} and taking $\delta\in (0,1/2)$ we obtain the thesis.
\qed

 Throughout the rest of this section we will assume Hypothesis $(H2)$ and further\\
\textbf{Hypothesis $(H3)$} For any $\varepsilon >0$ there is $C_\varepsilon >0$ such that 
\begin{itemize}
\item [(i)] $\mu \in W^{2,1}_{\rm loc}(\R^{N})$  and $\left |D_{i}\left (\frac{D_{j}\mu}{\mu}\right)\right |\leq \frac{\varepsilon}{|x|^{2}}+C_\varepsilon \left |\frac{\nabla \mu}{\mu}\right|$ for every $i,j=1,\dots,N$;
\item [(ii)] $\mu \in W^{3,1}_{\rm loc}(\R^{N})$ and $\left |D_{ij}\left (\frac{D_{k}\mu}{\mu}\right)\right |\leq \frac{\varepsilon}{|x|^{3}}+C_\varepsilon\left |\frac{\nabla \mu}{\mu}\right|$ for every $i,j,k=1,\dots,N$.
\end{itemize}
 
 In the sequel we need the following interpolation inequality.
\begin{proposition}\label{interp-ineq}
Assume that $N\ge 5$ and $\mu$ satisfies Hypotheses $(H2)$
and $(H3)(i)$.
Then for any $\varepsilon>0$ there exists $C_\varepsilon >0$ such that
\begin{equation}\label{eq:5-9}
\|\nabla u \|^2_{L^{2}_\mu}\leq \varepsilon \|D^2u\|^2_{L^{2}_\mu}+C_\varepsilon \|u\|^2_{L^{2}_\mu}
\end{equation}
for every $u\in C_c^\infty(\R^N)$.
\end{proposition}
{\sc Proof.}
It is well known that for every $\varphi\in H^2(\R^{N})$ we have
\[
\|\nabla \varphi\|^{2}_{L^{2}}\leq \varepsilon \|\Delta \varphi \|_{L^{2}}^{2}+\frac{C}{\varepsilon}\|\varphi\|_{L^{2}}^{2}.
\]
We take $\varphi=u\sqrt \mu$.
Then we have
\begin{align*}
&\|\nabla \varphi \|_{L^{2}}^{2}=\int_{\R^N} | \nabla u\, \sqrt \mu+u\nabla \left (\sqrt \mu\right) |^{2}dx
	=\|\nabla u\|_{L^{2}_\mu}^{2}+\int_{\R^N} \left(2u\sqrt \mu\, \nabla u\cdot  \nabla \left(\sqrt{\mu}\right)
		+u^{2}\left |\nabla \left(\sqrt \mu\right)\right|^{2}\right)dx\\
&\quad =	\|\nabla u\|_{L^{2}_\mu}^{2}+\int_{\R^N} u^{2}Ud\mu.
\end{align*}
On the other hand, since $\Delta\sqrt{\mu}=-\sqrt{\mu}U$, we have
\begin{align*}
&\|\Delta \varphi\|_{L^{2}}^{2}
	=\int_{\R^N} \left|\Delta u\sqrt \mu+2\nabla u\cdot\nabla \sqrt \mu+u\Delta \sqrt \mu \right|^{2}dx\\
&\quad	\leq C\left (\|\Delta u\|_{L^{2}_\mu}^{2} +\int_{\R^N} \left |\frac{\nabla \mu}{\mu}\cdot\nabla u\right |^{2}d\mu+\int_{\R^N} U^{2}u^{2}d\mu   \right).
\end{align*}
Estimating the second term in the right hand side by Proposition \ref{pr:prop-stima-drift}, one obtains
\[
\int_{\R^N} \left |\frac{\nabla \mu}{\mu}\cdot\nabla u\right |^{2}d\mu
	\leq C\left( \|\nabla u\|_{L^{2}_\mu}^{2}+\|D^{2}u\|_{L^{2}_\mu}^{2} \right).
\]
As regards the last term we observe that by 
Hypothesis $(H3)(i)$ 
 we have 
\[
U^{2}\leq C \left ( \left |\frac{\nabla \mu}{\mu}\right |^{2}   +\left |{\rm div}\left( \frac{\nabla \mu}{\mu}\right) \right | \right )^{2}\leq 
C\left( \left |\frac{\nabla \mu}{\mu}\right |^{4}+\left |\frac{\nabla \mu}{\mu}\right |^{2}+c_{1} \frac{1}{|x|^{4}}\right)
\]
for some $c_{1}<(C_0-1)^2=\left(\frac{N(N-4)}{4}\right)^{2}$.
By Proposition \ref{pr:prop-stima-drift} 
 we have $$\int_{\R^N} \left |\frac{\nabla \mu}{\mu}\right |^{2}u^{2}d\mu \leq C\|u\|_{H^{1}_{\mu}}^{2}.$$

Using twice Proposition \ref{pr:prop-stima-drift} and 
taking into account that by assumption we have
$\nabla \left | \frac{\nabla \mu}{\mu}  \right|\leq \frac{\varepsilon}{|x|^{2}}
	+C_\varepsilon \left | \frac{\nabla \mu}{\mu}   \right|$, it follows, by weighted Rellich's inequality, that
\begin{align*}
&\int_{\R^N} \left | \frac{\nabla \mu}{\mu}\right |^{4}u^{2}d\mu=
	\int_{\R^N} \left | \frac{\nabla \mu}{\mu}\right |^{2 }\left (\left | \frac{\nabla \mu}{\mu}\right | u\right)^{2}d\mu\leq 
		C\left( \left \| \frac{\nabla \mu}{\mu}u  \right\|_{L^{2}_\mu}^{2}
	+\left \|  \nabla \left | \frac{\nabla \mu}{\mu} u  \right|         \right\|_{L^{2}_\mu}^{2}  \right)\\
&\quad \leq C\left( \|u\|_{L^{2}_\mu}^{2}+\|\nabla u\|_{L^{2}_\mu}^{2}
	+\left \|\frac{\nabla \mu}{\mu}\cdot \nabla u\right \|_{L^{2}_\mu}^{2}
		+\int_{\R^N} \left |\nabla \left | \frac{\nabla \mu}{\mu}\right | \right |^{2}u^{2}d\mu\right)\\
&\quad \leq C\left( \|u\|_{L^{2}_\mu}^{2}+\|\nabla u\|_{L^{2}_\mu}^{2}+\|D^{2}u\|_{L^{2}_\mu}^{2}+
	\int_{\R^N} \varepsilon^{2}\frac{u^{2}}{|x|^{4}}d\mu +\int_{\R^N} \left |\frac{\nabla \mu}{\mu}\right |^{2}u^{2}d\mu \right)\\
&\quad \leq C\|u\|^{2}_{H_{\mu}^{2}}	.
\end{align*}
Then,
\[
\int_{\R^N} U^{2}u^{2}d\mu\leq C\|u\|^{2}_{H_{\mu}^{2}}.
\]
Finally,
\begin{align*}
\|\nabla u\|_{L^{2}_\mu}^{2}\leq \varepsilon \|u\|_{H^{2}_{\mu}}^{2}+\int_{\R^N} \left (\varepsilon U^{2}-U\right)u^{2}d\mu+C_\varepsilon \|u\|_{L^{2}_\mu}^{2}.
\end{align*}
Since $\varepsilon U^{2}-U\leq \frac{1}{\varepsilon}+2\varepsilon U^{2}$ we have
\begin{align*}
\|\nabla u\|_{L^{2}_\mu}^{2}\leq \varepsilon \|u\|_{H^{2}_{\mu}}^{2}+C_\varepsilon \|u\|_{L^{2}_\mu}^{2}.
\end{align*}
This implies \eqref{eq:5-9}. 
\qed

The following result concerns a kind of Calderon-Zygmund's inequality.
\begin{proposition}\label{weighted-Cal-Zyg}
  Assume that $N\geq5$ and $\mu$ satisfies Hypotheses $(H2)$  and $(H3)(i)$.


Then for all $u\in C_c^\infty(\R^{N})$ we have
\begin{equation}\label{eq:4-12}
\|D^{2}u\|_{L^{2}_\mu}\leq C\left( \| Lu\|_{L^{2}_\mu}+\|u\|_{L^{2}_\mu}\right).
\end{equation}
\end{proposition}
{\sc Proof.} 
For $u\in C_c^\infty(\R^N)$ we have
\begin{align*}
&\|D^{2}u\|^{2}_{L^{2}_\mu}=\sum_{i,j=1}^{N}\int_{\R^N} D_{ij}u\,D_{ij}u\,\mu dx=
	-\sum_{i,j=1}^{N}\int_{\R^N} \left(D_{iij}u\,D_{j}u\, \mu+D_{ij}u\, D_{j}u\, D_{i}\mu\right) dx\\
&\quad =\sum_{i,j=1}^{N}\left(\int_{\R^N} \left(D_{ii}u\, D_{jj}u\, \mu+D_{ii}u\, D_{j}u\, D_{j}\mu \right)dx -\int_{\R^N} D_{ij}u\, D_{j}u\, D_{i}\mu dx \right)\\
&\quad =\sum_{i,j=1}^{N}\left(\int_{\R^N} \left(D_{ii}u\, D_{jj}u\, \mu+D_{ii}u\, D_{j}u\, D_{j}\mu\right) dx 	
	+\int_{\R^N} \left(D_{i}u\, D_{j}u\, D_{ij}\mu+D_{i}u\, D_{jj}u\, D_{i}\mu\right) dx \right)\\
&\quad = \int_{\R^N} \left(( \Delta u )^{2}+2\Delta u \frac{\nabla \mu}{\mu}\cdot \nabla u+
	\left( \frac{D^{2}\mu}{\mu}\nabla u\right)\cdot\nabla u \right)d\mu.
\end{align*}
We observe that
\[
\left(  \frac{D^{2}\mu}{\mu}\nabla u\right)\cdot\nabla u =\left( \frac{\nabla \mu}{\mu} \cdot \nabla u\right)^{2}
	-(D^{2}\phi\nabla u)\cdot\nabla u ,
\]
where $\phi$ is such that $\mu=e^{-\phi}$. Then finally
\begin{equation*}
\|D^{2}u\|^{2}_{L^2_\mu}=\int_{\R^N} \left((Lu)^{2}-\left( D^{2}\phi\nabla u\right)\cdot\nabla u  \right)d\mu.
\end{equation*}
We estimate now the last term in the inequality. By $(H3)(i)$ we obtain for
any $\varepsilon >0$ there is $C_\varepsilon >0$ such that 
\[
|\left(D^{2}\phi\nabla u\right)\cdot\nabla u | \leq \varepsilon \frac{|\nabla u|^{2}}{|x|^{2}}+C_\varepsilon\left |\frac{\nabla \mu}{\mu}\right | |\nabla u|^{2}.
\]
So, using \eqref{weigthed_Hardy}, we have
\begin{align*}
& -\int_{\R^N}  \left(D^{2}\phi\nabla u\right)\cdot\nabla u  d\mu\leq
\varepsilon \int_{\R^N}  \frac{|\nabla u|^{2}}{|x|^{2}}d\mu
	+C_{\varepsilon} \int_{\R^N}  \left |\frac{\nabla \mu}{\mu}\right |    |\nabla u |^{2} d\mu\\
&\quad \leq C_0^{-1}\varepsilon \|D^{2}u\|^2_{L^{2}_\mu}+C_0^{-1} \varepsilon c_{\mu} \|\nabla u\|^2_{L^{2}_\mu}
		+C_{\varepsilon}\delta \int_{\R^N} \left |\frac{\nabla \mu}{\mu}\right | ^{2}   |\nabla u |^{2} d\mu
		+C_{\varepsilon,\delta}\int_{\R^N} |\nabla u |^{2} d\mu\\
&\quad \leq (C_0^{-1}\varepsilon +\delta C_{\varepsilon})\|D^{2}u\|^2_{L^{2}_\mu}+C_{\varepsilon,\delta}\|\nabla u\|^2_{L^2_{\mu}}=\varepsilon'\|D^{2}u\|^2_{L^{2}_\mu}+C_{\varepsilon'}\|\nabla u\|^2_{L^2_{\mu}}.
\end{align*}
Using the weighted interpolation inequality \eqref{eq:5-9} we have for every $\delta >0$ there is $C_\delta >0$ such that
\begin{eqnarray*}
\|D^{2}u\|^{2}_{L^{2}_\mu} &\leq & \|Lu\|_{L^{2}_\mu}^{2}+
\varepsilon' \|D^{2}u\|^2_{L^{2}_\mu}+C_{\varepsilon'}\|\nabla u\|^2_{L^2_\mu}\\
&\leq & \|Lu\|^2_{L^2_\mu}+\varepsilon' \|D^{2}u\|^2_{L^{2}_\mu}+
C_{\varepsilon'}
	\left( \delta \|D^{2}u\|_{L^{2}_\mu}^{2}+C_{\delta}\|u\|^{2}_{L^{2}_\mu} \right).
\end{eqnarray*}
So, the assertion follows by choosing $\varepsilon'$ and $\delta$ such that 
$\varepsilon'+\delta C_{\varepsilon'}<\frac{1}{2}$.

 \qed

 Applying Propositions \ref{pr:prop-stima-drift}, \ref{interp-ineq} and \ref{weighted-Cal-Zyg}   one obtains the following characterization of $D(L)$.
 
 \begin{theorem}\label{th:characterization-1}
Assume that $N\ge 5$ and $\mu$ satisfies  Hypotheses $(H1),(H2)$ and $(H3)(i)$
then
\[
D(L)=H^{2}_\mu(\R^{N}).
\]
 
\end{theorem}
\begin{proof}
It follows from Propositions \ref{weighted-Cal-Zyg},  \ref{interp-ineq} and   \ref{pr:prop-stima-drift} that
\begin{equation}\label{eq:L-H-mu}
C^{-1}\|u\|_{H^2_\mu}\le \|Lu\|_{L^2_\mu}+\|u\|_{L^2_\mu}\le C\|u\|_{H^2_\mu}
\end{equation}
for all $u\in C_c^\infty(\R^N)$ and some constant $C>0$. So, since $C_c^\infty(\R^N)$ is a core for $L$, we deduce that $D(L)\subset H^2_\mu(\R^N)$.

Now, if $u\in H^2_\mu(\R^N)$, then, integrating by part, we obtain
$$\int_{\R^N}\nabla u\cdot \overline{\nabla \varphi}\,d\mu=-\int_{\R^N}(\Delta u+\frac{\nabla \mu}{\mu}\cdot \nabla u)\overline{\varphi }\,d\mu$$
for all $\varphi\in C_c^\infty(\R^N)$. Thus, using Proposition \ref{pr:prop-stima-drift}, we have $(\Delta u+\frac{\nabla \mu}{\mu}\cdot \nabla u)\in L^2_\mu(\R^N)$ and hence $u\in D(L)$.
\end{proof}

\begin{remark}
Observe that since $C_c^\infty(\R^N)$ is a core for $L$ and $D(L)=H^2_\mu(\R^N)$, by equivalence of the norms \eqref{eq:L-H-mu}    it follows that  $C_c^\infty(\R^N)$ is dense in the weighted Sobolev space $H^2_\mu(\R^N)$.
\end{remark}

As a consequence of the above theorem and the weighted Rellich inequality \eqref{weighted-Rellich0} one obtains the following.
\begin{theorem}
Assume $N\geq5$ and Hypotheses (H1), (H2), (H3)(i). For every $0\leq V(x)\leq\frac{c}{|x|^4}$ with $c<\left(\frac{N(N-4)}{4}\right)^2$, the perturbed operator $(-A+V,D(A))$ is the generator of an analytic $C_0$-semigroup on $L^2_\mu(\R^N)$. If  instead, $c=\left(\frac{N(N-4)}{4}\right)^2$, 
a suitable extension of $-A+\frac{c}{|x|^4}$  is the generator of
an analytic $C_0$-semigroup on $H^1_\mu(\R^N)$.
\end{theorem}
\begin{proof}
Let us first consider the case $c<\left(\frac{N(N-4)}{4}\right)^2$ and the sesquilinear form
\begin{align*}
a_{L,V}(u,v)&=\int_{\R^N}Lu\overline{Lv}d\mu-\int_{\R^N} Vu\overline{v}d\mu,\\ D(a_{L,V})&=D(L)
\end{align*}
associated to $A-V$. Since by Theorem \ref{th:characterization-1} $D(L)=H^2_\mu(\R^N)$, it follows by \eqref{weighted_Rellich} that $D(a_{L,V})=H^2_\mu(\R^N)$ and $D(A-V)=D(A)$. Now, since $c<(C_0-1)^2$, there is $\varepsilon >0$ such that $b_\varepsilon :=c((C_0-1)^2-\varepsilon)^{-1}<1$. So, applying \eqref{weighted-Rellich0} we obtain
$$a_{L,V}(u,u)+\frac{((C_0-1)C_1)^2}{\varepsilon}\int_{\R^N}|u |^2d\mu\ge (1-b_\varepsilon)\int_{\R^N}|Lu |^2d\mu$$
for all $u\in D(L)$ and some constant $C_1>0$. Thus $a_{L,V}$ with domain $D(L)$ is closed and quasi-accretive. So the first assertion follows, since $a_{L,V}$ is continuous and densely defined.

As regards the limit case $V=\frac{c}{|x|^4}$ and $c=\left(\frac{N(N-4)}{4}\right)^2$, we have, by \eqref{weighted_Rellich}, 
$$a_{L,V}(u,u)+C\|u\|_{H^1_\mu}\ge 0\quad \hbox{\ for all }u\in D(L).$$
So, since $A$ is symmetric, by \cite[Lemma 1.29]{ouh}, one obtains that $a_{L,V}$ is closable and a suitable extension of $-A+V$ associated to the closure of $a_{L,V}$ is the generator of an analytic $C_0$-semigroup on $H^1_\mu(\R^N)$.

\end{proof}
We note now that Proposition \ref{pr:prop-stima-drift} allows us to prove the following useful estimates.
\begin{proposition}\label{prop:more-rellich}
Let $N\geq 5$ and $\mu$ satisfies Hypotheses $(H1)$ and $(H2)$. Then
\begin{itemize}
\item[1.]
\begin{equation}\label{hardy-stima2}
\int_{\R^N} \frac{|u(x)|^{2}}{|x|^{4}}d\mu \leq C\|u\|^{2}_{H^{2}_\mu}
\end{equation}
for all $u\in H^2_\mu(\R^N)$ and if we assume in addition that $(H3)(i)$ holds then
\begin{align}\label{hardy-1derivative}
\int_{\R^N}\frac{|\nabla u(x)|^2}{|x|^4}d\mu
\leq C\left(\int_{\R^N}|\nabla Lu(x)|^2d\mu +\|u\|^2_{H^2_\mu}\right)
\end{align}
for all $u\in D(L)$ with $Lu\in H^1_\mu(\R^N)$.
\item[2.] If $N\ge 7$, then 
\begin{equation}\label{hardy-stima3}
\int_{\R^N} \frac{|u(x)|^{2}}{|x|^{6}}d\mu \leq C\left(\int_{\R^N}|L\nabla u(x)|^2d\mu +\|u\|^{2}_{H^{2}_\mu}\right)
\end{equation}
for any $u\in H_\mu^3(\R^N)$.
\end{itemize}
We stress that all the constants $C$ appearing in \eqref{hardy-stima2}, \eqref{hardy-1derivative} and in \eqref{hardy-stima3} are independent of $u$.
\end{proposition}
{\sc Proof.}
Estimate \eqref{hardy-stima2} follows from the weighted Rellich inequality \eqref{weighted_Rellich}, \eqref{eq:L-H-mu} and Theorem \ref{th:characterization-1}.
As regards the estimate \eqref{hardy-stima3} we apply the weighted Hardy inequality \eqref{weigthed_Hardy} to the function $\frac{u}{|x|^{2}+\delta}$ for $\delta >0$. Then, by Young's inequality, for any $\varepsilon >0$ there is $C_\varepsilon>0$ such that 
\begin{align*}
&C_{0}\int_{\R^{N}}\frac{1}{|x|^{2}} \left( \frac{u}{|x|^{2}+\delta} \right)^{2}d\mu\leq 
		\int_{\R^{N}} \left |\nabla \left( \frac{u}{|x|^{2}+\delta}\right)\right |^{2}d\mu+C\int_{\R^{N}} \frac{|u|^{2}}{(|x|^{2}+\delta)^2}d\mu\\
&\quad =	\int_{\R^{N}} \left | \frac{\nabla u}{|x|^{2}+\delta}-2\frac{x}{(|x|^{2}+\delta)^2}u \right |^{2}d\mu
	+C\int_{\R^{N}} \frac{|u|^{2}}{(|x|^{2}+\delta)^2}d\mu\\	
&\quad \leq (4+\varepsilon)\int _{\R^{N}}\frac{|u|^{2}}{(|x|^{2}+\delta)^3}d\mu+C_{\varepsilon}\int_{\R^{N}} \frac{|\nabla u|^{2}}{(|x|^{2}+\delta)^2}d\mu+C\int_{\R^{N}} \frac{|u|^{2}}{(|x|^{2}+\delta)^2}d\mu \\	
&\quad \leq  (4+\varepsilon)\int _{\R^{N}}\frac{|u|^{2}}{|x|^2(|x|^{2}+\delta)^2}d\mu +C_{\varepsilon}\int_{\R^{N}} \frac{|\nabla u|^{2}}{|x|^{4}}d\mu+C\int_{\R^{N}} \frac{|u|^{2}}{|x|^{4}}d\mu.
\end{align*}
Applying \eqref{weighted_Rellich} we have
\begin{eqnarray*}
\int _{\R^{N}}\frac{|\nabla u|^{2}}{|x|^{4}}d\mu &=&
\sum_{i=1}^{N}\int _{\R^{N}}\frac{|D_{i} u|^{2}}{|x|^{4}}d\mu \\
&\leq & \frac{1}{(C_{0}-1^{2})}\sum_{i=1}^{N} \int_{\R^{N}} |LD_{i}u|^{2}d\mu+C\|D_{i}u\|_{H^{1}_{\mu}}\\
&\leq &
C\left (\int_{\R^{N}}|L\nabla u|^{2}d\mu+\|u\|^2_{H^{2}_{\mu}}\right).
\end{eqnarray*}
So, using \eqref{hardy-stima2}, we deduce
\[
\left( C_{0}-(4+\varepsilon)\right)\int_{\R^{N}}\frac{|u|^{2}}{|x|^{2}(|x|^2+\delta)^2}\leq C\left(\int_{\R^N}|L\nabla u|^2d\mu +\|u\|^{2}_{H^{2}_\mu}\right).
\]
We observe that for $N>6$ and $\varepsilon$ small enough we have $C_{0}-(4+\varepsilon)>0$. Then 
\eqref{hardy-stima3} follows by applying Fatou's lemma.

To prove \eqref{hardy-1derivative} we show first that
\begin{equation}
\label{stima_Rellich_der_prima}
\int_{\R^N}\frac{|\nabla \varphi|^2}{|x|^2(|x|^2+\delta)}d\mu
\leq -C\int_{\R^N}\frac{ \nabla \varphi\cdot\nabla L\varphi }{|x|^2+\delta}d\mu+M\|\varphi\|_{H^2_\mu}
\end{equation}
for all $\varphi\in C_c^\infty(\R^N)$ and some positive constants $C,\,M$ independent of $\varphi$.
\\
Using $(H3)(i)$ and applying \eqref{weighted_hardy_1} and \eqref{conto_grad_weighted_hardy}, we obtain 
\begin{eqnarray*}
& & (C_0-1)\int_{\R^N}\frac{|D_k \varphi|^2}{|x|^2(|x|^2+\delta)}d\mu \le 
-\int_{\R^N}LD_k\varphi \frac{D_k\varphi}{|x|^2+\delta}d\mu +C_1\int_{\R^N}\frac{|D_k\varphi|^2}{|x|^2+\delta}d\mu \\
&=& -\int_{\R^N}D_kL\varphi \frac{D_k\varphi}{|x|^2+\delta}d\mu +\int_{\R^N}  \nabla\varphi\cdot\nabla\left(\frac{D_k\mu}{\mu}\right) \frac{D_k\varphi}{|x|^2+\delta}d\mu +C_1\int_{\R^N}\frac{|D_k\varphi|^2}{|x|^2+\delta}d\mu \\
&\le & -\int_{\R^N}D_kL\varphi \frac{D_k\varphi}{|x|^2+\delta}d\mu +\varepsilon\int_{\R^N}\frac{|\nabla \varphi|^2}{|x|^2(|x|^2+\delta)}d\mu+C_\varepsilon \int_{\R^N}\frac{|\nabla \varphi|^2}{|x|^2+\delta}\left|\frac{\nabla \mu}{\mu}\right|d\mu \\
&  &+C_1\theta \int_{\R^N}\frac{|\nabla \varphi|^2}{(|x|^2+\delta)^2}d\mu+\frac{C_1}{4\theta}\int_{\R^N}|\nabla \varphi|^2d\mu \\
&\le & -\int_{\R^N}D_kL\varphi \frac{D_k\varphi}{|x|^2+\delta}d\mu +\varepsilon\int_{\R^N}\frac{|\nabla \varphi|^2}{|x|^2(|x|^2+\delta)}d\mu +
C_\varepsilon \eta \int_{\R^N}\frac{|\nabla \varphi|^2}{(|x|^2+\delta)^2}d\mu \\
&  & +\frac{C_\varepsilon}{4\eta}\int_{\R^N}|\nabla \varphi|^2\left|\frac{\nabla \mu}{\mu}\right|^2d\mu +C_1\theta \int_{\R^N}\frac{|\nabla \varphi|^2}{(|x|^2+\delta)^2}d\mu+\frac{C_1}{4\theta}\int_{\R^N}|\nabla \varphi|^2d\mu.
\end{eqnarray*}
for any $\varepsilon ,\,\theta,\,\eta >0$ and $\varphi\in C_c^\infty(\R^N)$. Choosing $\varepsilon ,\,\theta$ and $\eta$ small enough we obtain
$$\int_{\R^N}\frac{|\nabla \varphi|^2}{|x|^2(|x|^2+\delta)}d\mu \le -C\int_{\R^N}\frac{  \nabla \varphi \cdot\nabla L\varphi }{|x|^2+\delta}d\mu+M\left(\int_{\R^N}|\nabla \varphi|^2\left|\frac{\nabla \mu}{\mu}\right|^2d\mu+\int_{\R^N}{|\nabla \varphi|^2}d\mu\right)$$
and so, by Proposition \ref{pr:prop-stima-drift}, we obtain \eqref{stima_Rellich_der_prima}. \\
We now prove that 
\ref{pr:prop-stima-drift} remains true for all $u\in D(L)$ such that $Lu\in H^1_\mu(\R^N)$. For such function $u$, there is $\varphi_n\in C_c^\infty(\R^N)$ such that $\lim_{n\to \infty}\varphi_n =u$ and $\lim_{n\to \infty}L\varphi_n=Lu$ in $L^2_\mu(\R^N)$. Since the graph norm of $L$ and the $H^2_\mu$-norm are equivalent, we deduce that $\lim_{n\to \infty}\|\varphi_n-u\|_{H^2_\mu}=0$. So, it remains to show that
$$\lim_{n\to \infty}\int_{\R^N}\frac{  \nabla \varphi_n \cdot\nabla 
L\varphi_n }{|x|^2+\delta}d\mu= \int_{\R^N}\frac{  \nabla u \cdot\nabla Lu }{|x|^2+\delta}d\mu.$$ 
To this purpose, integrating by parts, recalling the definition of $L$ and taking into account that $Lu\in H^1_\mu(\R^N)$, we have
\begin{eqnarray*}
\lim_{n\to \infty}\int_{\R^N}\frac{  \nabla \varphi_n \cdot\nabla 
L\varphi_n }{|x|^2+\delta}d\mu &=&\lim_{n\to \infty}\left(-\int_{\R^N}\frac{|L\varphi_n|^2}{|x|^2+\delta}d\mu+2\int_{\R^N} L\varphi_n\frac{  \nabla\varphi_n \cdot x }{(|x|^2+\delta)^2}d\mu\right)\\
&=& -\int_{\R^N}\frac{|Lu|^2}{|x|^2+\delta}d\mu+2\int_{\R^N} Lu\frac{  \nabla u \cdot x }{(|x|^2+\delta)^2}d\mu \\
&=& \int_{\R^N}\frac{  \nabla u\cdot \nabla 
Lu }{|x|^2+\delta}d\mu ,
\end{eqnarray*}
since the functions $x\mapsto (|x|^2+\delta)^{-1}$ and $x\mapsto x (|x|^2+\delta)^{-2}$ are bounded, where the last equality follows from the definition of $L$ and the fact that $Lu\in H^1_\mu(\R^N)$. Thus, Young's inequality gives, for any $\varepsilon >0$ there is $C_\varepsilon>0$ such that
\begin{eqnarray*}
\int_{\R^N}\frac{|\nabla u|^2}{|x|^2(|x|^2+\delta)}d\mu &\le & -C\int_{\R^N}\frac{  \nabla u\cdot\nabla  Lu }{|x|^2+\delta}d\mu +M\|u\|^2_{H^2_\mu}\\
&\le &\varepsilon \int_{\R^N}\frac{|\nabla u|^2}{|x|^2(|x|^2+\delta)}d\mu +C_\varepsilon \int_{\R^N}|\nabla Lu|^2d\mu +M\|u\|^2_{H^2_\mu}
\end{eqnarray*}
for all $u\in D(L)$ with $Lu\in H^1_\mu(\R^N)$. Choosing $\varepsilon=\frac12$ we get
\begin{align*}
\int_{\R^N}\frac{|\nabla u|^2}{|x|^2(|x|^2+\delta)}d\mu
\le &\varepsilon \int_{\R^N}\frac{|\nabla u|^2}{|x|^2(|x|^2+\delta)}d\mu +C_\varepsilon \int_{\R^N}|\nabla Lu|^2d\mu +M\|u\|^2_{H^2_\mu},
\end{align*}
for all $u\in D(L)$ with $Lu\in H^1_\mu(\R^N)$ and some constants $C,\,M>0$ independent of $u$. Letting $\delta \to 0$ in the above estimate, Fatou's lemma gives \eqref{hardy-1derivative}. This ends the proof.
\qed

The following results shows that $D(L^2)\subset H^3_\mu(\R^N)$.
\begin{corollary}\label{cor-regularity-H3}
Assume that $N\ge 7$ and $\mu$ satisfies $(H1),(H2)$ and $(H3)(i)$. Then
$$H^3_\mu(\R^N) =\{u\in H_\mu^2(\R^N): Lu\in H^1_\mu(\R^N)\}.$$ 
In particular, $D(L^2)\subset H^3_\mu(\R^N)$.
\end{corollary}
{\sc Proof.}
Let $u\in H_\mu^3(\R^N)$. Then, by $(H3)(i)$, Proposition \ref{pr:prop-stima-drift} and \eqref{hardy-stima2}, we deduce that $\frac{\nabla\mu}{\mu}\cdot \nabla u\in H^1_\mu(\R^N)$ and 
$$\left\|\frac{\nabla\mu}{\mu}\cdot \nabla u\right\|\le C\|u\|_{H^3_\mu}.$$
So, $Lu\in H^1_\mu(\R^N)$. \\
Conversely, let us consider $u\in H^2_\mu(\R^N)$ with $Lu\in H^1_\mu(\R^N)$. Since, by Theorem \ref{th:characterization-1}, $D(L)=H^2_\mu(\R^N)$, we have to prove that $D_ku\in D(L)$ for any $k=1,\ldots ,N$. Integrating by parts, we obtain
$$\int_{\R^N}\nabla(D_ku)\cdot \nabla\varphi \,d\mu=-\int_{\R^N}(D_kLu) \varphi \,d\mu+\int_{\R^N}\nabla u\cdot \nabla\left(\frac{D_k\mu}{\mu}\right)\varphi \,d\mu$$
for all $\varphi\in C_c^\infty(\R^N)$. So, since $C_c^\infty(\R^N)$ is is a core for $L$, to show that $D_ku\in D(L)$ it suffices to prove that  $\nabla u\cdot \nabla\left(\frac{D_k\mu}{\mu}\right)\in L^2_\mu(\R^N)$, which can be obtained by using $(H3)(i)$, Proposition \ref{pr:prop-stima-drift} and \eqref{hardy-1derivative}. Moreover,
\begin{equation}\label{eq-comut}
LD_ku=D_kLu-\nabla\left(\frac{D_k\mu}{\mu}\right)\cdot \nabla u.
\end{equation}
\qed

For the complete description of $D(A)$ we need the following versions of   higher order Rellich's inequalities.
\begin{lemma}\label{lem:H4}
Assume that $N\ge 7$ and $\mu$ satisfies $(H1)-(H3)$. Then
\begin{equation}\label{hardy-2derivative}
\int_{\R^N}\frac{|\nabla u|^2}{|x|^6}d\mu
\leq C_1\left(\int_{\R^N}|Au|^2d\mu +\|u\|^2_{H^3_\mu}\right)
\end{equation}
and 
\begin{equation}\label{hardy-3derivative}
\int_{\R^N}\frac{|D^2 u|^2}{|x|^4}d\mu
\leq C_2\left(\int_{\R^N}|Au|^2d\mu +\|u\|_{H^3_\mu}\right)
\end{equation}
for any $u\in D(L^2)$ and some constants $C_1,C_2>0$.
\end{lemma}
\begin{proof}
Similar computations as in the proof of Proposition \ref{prop:more-rellich} yield
\begin{equation}\label{eq:quarta-stima}
\int_{\R^N}\frac{|\nabla \varphi|^2}{|x|^2(|x|^2+\delta)^2}d\mu
\leq -C\int_{\R^N}\frac{  \nabla \varphi\cdot\nabla L\varphi }{(|x|^2+\delta)^2}d\mu+M\left(\int_{\R^N}\frac{|\nabla \varphi|^2}{|x|^2+\delta}\left|\frac{\nabla \mu}{\mu}\right|^2d\mu+\int_{\R^N}\frac{|\nabla \varphi|^2}{(|x|^2+\delta)^2}d\mu\right)
\end{equation}
for all $\varphi\in C_c^\infty(\R^N)$ and some constants $C,\,M>0$ independent of $\varphi$. To show that \eqref{eq:quarta-stima} remains valid for $u\in D(L^2)$, as in the proof of Proposition \ref{prop:more-rellich}, we consider $\varphi_n\in C_c^\infty(\R^N)$ such that $\lim_{n\to \infty}\|\varphi_n-u\|_{H^2_\mu}=0$. Using \eqref{weigthed_Hardy}, it is easy to see
that 
\begin{align*}
\lim_{n\to\infty}\int_{\R^N}\frac{|\nabla \varphi_n|^2}{|x|^2(|x|^2+\delta)^2}d\mu & = \int_{\R^N}\frac{|\nabla u|^2}{|x|^2(|x|^2+\delta)^2}d\mu \,\,\hbox{\ and } \\
\lim_{n\to\infty}\int_{\R^N}\frac{|\nabla \varphi_n|^2}{(|x|^2+\delta)^2}d\mu  & = \int_{\R^N}\frac{|\nabla u|^2}{(|x|^2+\delta)^2}d\mu .
\end{align*} 
To compute $\lim_{n\to\infty}\int_{\R^N}\frac{|\nabla \varphi_n|^2}{|x|^2+\delta}\left|\frac{\nabla \mu}{\mu}\right|^2d\mu$
we remark that, by Proposition \ref{pr:prop-stima-drift}, we have
\begin{eqnarray*}
\int_{\R^N}\frac{|D_k v|^2}{|x|^2+\delta}\left|\frac{\nabla \mu}{\mu}\right|^2d\mu
&\leq & C\left(\int_{\R^N}\left|\nabla\left(\frac{D_kv}{(|x|^2+\delta)^{1/2}}\right)\right|^2d\mu+\int_{\R^N}\frac{|D_kv|^2}{|x|^2+\delta}d\mu\right)\\
&\le & C\left(\int_{\R^N}\frac{|D^2v|^2}{|x|^2+\delta}d\mu+\int_{\R^N}\frac{|\nabla v|^2}{(|x|^2+\delta)^2}d\mu+\int_{\R^N}\frac{|\nabla v|^2}{|x|^2+\delta}d\mu \right)
\end{eqnarray*}
for any $v\in H^2_\mu(\R^N)$ and $k=1,\ldots ,N$. Hence,
\begin{align*}
\lim_{n\to\infty}\int_{\R^N}\frac{|\nabla (\varphi_n-u)|^2}{|x|^2+\delta}\left|\frac{\nabla \mu}{\mu}\right|^2d\mu
\leq &\ C\lim_{n\to\infty}\bigg(\int_{\R^N}\frac{|D^2(\varphi_n-u)|^2}{(|x|^2+\delta)^2}d\mu \\
& +\int_{\R^N}\frac{|\nabla (\varphi_n-u)|^2}{(|x|^2+\delta)^2}d\mu
+\int_{\R^N}\frac{|\nabla(\varphi_n-u)|^2}{|x|^2+\delta}d\mu \bigg)=0.
\end{align*}
The last limit 
$$\lim_{n\to\infty}\int_{\R^N}\frac{  \nabla \varphi_n\cdot\nabla L\varphi_n }{(|x|^2+\delta)^2}d\mu=\int_{\R^N}\frac{  \nabla u\cdot\nabla Lu }{(|x|^2+\delta)^2}d\mu$$
follows as in the proof of \eqref{hardy-1derivative}. Thus, \eqref{eq:quarta-stima} holds true for all $u\in D(L^2)$.

Now, using $(H3)(i)$, \eqref{weigthed_Hardy}, Proposition \ref{pr:prop-stima-drift} and \eqref{hardy-stima2}, we can see that
$$\int_{\R^N}\frac{|\nabla u|^2}{|x|^2+\delta}\left|\frac{\nabla \mu}{\mu}\right|^2d\mu+\int_{\R^N}\frac{|\nabla u|^2}{(|x|^2+\delta)^2}d\mu \le M\|u\|_{H^3_\mu}$$ 
holds for all $u\in D(L^2)$ and some constant $M>0$ independent of $u$, since, by Corollary \ref{cor-regularity-H3}, we know that $u\in H^3_\mu(\R^N)$ whenever $u\in D(L^2)$. Thus, by Young's inequality,
\begin{eqnarray*}
\int_{\R^N}\frac{|\nabla u|^2}{|x|^2(|x|^2+\delta)^2}d\mu
&\leq & -C\int_{\R^N}\frac{  \nabla u\cdot\nabla Lu }{(|x|^2+\delta)^2}d\mu+M\|u\|_{H^3_\mu}\\
&\le & \varepsilon \int_{\R^N}\frac{|\nabla u|^2}{(|x|^2+\delta)^3}d\mu +C_\varepsilon \int_{\R^N}\frac{|\nabla Lu|^2}{|x|^2+\delta} d\mu +M\|u\|_{H^3_\mu}.
\end{eqnarray*}
Thus, by choosing $\varepsilon$ small enough and \eqref{weigthed_Hardy}, we have
\begin{eqnarray*}
\int_{\R^N}\frac{|\nabla u|^2}{|x|^2(|x|^2+\delta)^2}d\mu &\le & C \int_{\R^N}\frac{|\nabla Lu|^2}{|x|^2+\delta} d\mu +M\|u\|_{H^3_\mu}\\
 &\le & C \int_{\R^N}\frac{|\nabla Lu|^2}{|x|^2} d\mu +M\|u\|_{H^3_\mu}\\
 &\le & C_1 \left(\int_{\R^N}|D^2Lu|^2d\mu+\int_{\R^N}|\nabla Lu|^2d\mu\right)+M\|u\|_{H^3_\mu}\\
 &\le &  C_2 \left(\int_{\R^N}|Au|^2d\mu +\|u\|_{H^3_\mu}\right).
 \end{eqnarray*}
 Thus, \eqref{hardy-2derivative} follows by applying Fatou's lemma.
 
 Finally, Estimate \eqref{hardy-3derivative} follows by applying \eqref{hardy-2derivative}, \eqref{hardy-1derivative}, \eqref{eq-comut} and taking into account assumption $(H3)$.

\end{proof}

We can now give a characterization of $D(A)$.

\begin{theorem}\label{th:characterization-2}
Assume that $N\geq7$ and 
$\mu$  satisfies Hypotheses  $(H1)$, $(H2)$ and $(H3)$
then 
\[
D(A)=H^{4}_{\mu}(\R^{N}).
\]
\end{theorem}
\begin{proof}
By Theorem \ref{th:characterization-1} we know that $L$ with domain $H_\mu^2(\R^N)$ is selfadjoint and hence generates an analytic semigroup of angle $\frac{\pi}{2}$. Moreover, since $L$ is dissipative it follows that
\begin{equation}\label{eq:analytic}
\|T(z)\|_{\mathcal{L}(L^2_\mu(\R^N))}\le 1,\quad \hbox{\ for all } z\in \mathbb{C} \hbox{\ with }\Re z>0,
\end{equation}
(cf. \cite[Example 3.7.5]{ABHN}). Using \eqref{eq:analytic} we deduce from \cite[Proposition 3.9.1 and Remark 3.9.3]{ABHN} that $iL$ generates a $C_0$-group on $L^2_\mu(\R^N)$. Thus it follows from \cite[Corollary 3.7.15]{ABHN} that $(iL)^2$ generates an analytic contraction $C_0$-semigroup of angle $\frac{\pi}{2}$. Since $D(L^2)\subset D(A)$ and both $-L^2$ and $-A$ are generators, it follows that $A=L^2$. Hence,
$$D(A)=\{u\in H_\mu^2(\R^N);\,Lu\in H_\mu^2(\R^N)\}.$$

Let us prove now that $D(A)=H_\mu^4(\R^N)$.

To this purpose let us observe first, by $(H3)$, we have
\begin{equation}\label{eq:D(A)-1}
\int_{\R^N} \left | D_{i} \left( \frac{\nabla \mu}{\mu} \right)\right|^{2}\varphi^{2}d\mu
\leq C\left(  \varepsilon\int_{\R^N} \frac{1}{|x|^{4}}\varphi^{2}d\mu
	+\int_{\R^N}  \left |\frac{\nabla \mu}{\mu}\right|^{2}\varphi^{2}d\mu\right).
\end{equation}
Hence, using the weighted Rellich inequality \eqref{weighted_Rellich} and Proposition \ref{pr:prop-stima-drift} we deduce that
$\left | D_{i} \left( \frac{\nabla \mu}{\mu} \right)\right|\varphi \in L^{2}_{\mu}(\R^{N})$ for $\varphi \in H_\mu^2(\R^N)$. By the same arguments we have 
$\left | D_{ij} \left( \frac{\nabla \mu}{\mu} \right)\right|\psi \in L^{2}_{\mu}(\R^{N})$ for all $\psi\in H_\mu^3(\R^N)$, since
\begin{equation}\label{eq:D(A)-2}
\int_{\R^N} \left | D_{ij} \left( \frac{\nabla \mu}{\mu} \right)\right|^{2}\psi^{2}d\mu
\leq C\left(  \varepsilon\int_{\R^N} \frac{1}{|x|^{6}}\psi^{2}d\mu
	+\int_{\R^N}  \left |\frac{\nabla \mu}{\mu}\right|^{2}\psi^{2}d\mu\right).
\end{equation}
Thus, observing that, for $u\in H_\mu^4(\R^N)$,
\begin{equation}\label{eq:dom-derive1}
D_{k}Lu=LD_{k}u+D_{k}\left( \frac{\nabla \mu}{\mu} \right)\cdot \nabla u
\end{equation}
and
\begin{equation}\label{eq:dom-derive2}
D_{hk}Lu=LD_{hk}u+D_{hk}\left( \frac{\nabla \mu}{\mu} \right)\cdot \nabla u 
	+D_{h}\left( \frac{\nabla \mu}{\mu} \right)\cdot \nabla (D_{k}u)
	+D_{k}\left( \frac{\nabla \mu}{\mu} \right)\cdot \nabla (D_{h}u),
\end{equation}
it follows from \eqref{eq:D(A)-1}, \eqref{eq:D(A)-2}, \eqref{hardy-stima2},   \eqref{hardy-stima3},  and Proposition \ref{pr:prop-stima-drift} that $D_{k}Lu$ and $D_{hk}Lu$ belong to $L^2_\mu(\R^N)$ for all $h,\,k\in \{1,2,\ldots ,N\}$. Thus $Lu\in H_\mu^2(\R^N)$ and hence $H_\mu^4(\R^N)\subset D(A)$. 

For the other inclusion, since by Theorem \ref{th:characterization-1}, $D(L)=H^2_\mu(\R^N)$, we have to prove that $D_{hk}u\in D(L)$ for any $h,k=\,\ldots ,N$ whenever $u\in D(A)$. To this purpose let us consider $u\in D(A)$. Integrating by parts, we get
\begin{eqnarray*}
\int_{\R^N}\nabla(D_{hk}u)\cdot \nabla \varphi \,d\mu &=& -\int_{\R^N}D_{hk}Lu \varphi\,d\mu +\int_{R^N} D_{hk}\left( \frac{\nabla \mu}{\mu} \right)\cdot \nabla u\varphi\,d\mu \\
& & +\int_{\R^N} D_{h}\left( \frac{\nabla \mu}{\mu} \right)\cdot \nabla (D_{k}u)
	\varphi \,d\mu +\int_{\R^N} D_{k}\left( \frac{\nabla \mu}{\mu} \right)\cdot\nabla (D_{h}u)\varphi \,d\mu
\end{eqnarray*}
for all $\varphi\in C_c^\infty(\R^N)$. So, as in the proof of Corollary \ref{cor-regularity-H3}, we have only to prove that $D_{hk}\left( \frac{\nabla \mu}{\mu} \right)\cdot\nabla u $ and $D_{k}\left( \frac{\nabla \mu}{\mu} \right)\cdot\nabla (D_{h}u)$ belong to $L_\mu^2(\R^N)$. \\
Using $(H3)(ii)$, \eqref{hardy-2derivative} and Proposition \ref{pr:prop-stima-drift}, one can see that $D_{hk}\left( \frac{\nabla \mu}{\mu} \right)\cdot \nabla u\in L_\mu^2(\R^N)$. By $(H3)(i)$, \eqref{hardy-3derivative}, Proposition \ref{pr:prop-stima-drift} and Corollary \ref{cor-regularity-H3}, one deduces that
$D_{k}\left( \frac{\nabla \mu}{\mu} \right)\cdot\nabla (D_{h}u)\in L^2_\mu(\R^N)$. Thus, $D_{hk}u\in D(L)$ for any $h,k=1,\ldots ,N$ and hence $u\in H^4_\mu(\R^N)$. This ends the proof.

\end{proof}

From the above proof one can deduce that the graph norm of $A$ is equivalent to the $H^4_\mu(\R^N)$-norm.
\begin{remark}
One has
$$C^{-1}\|u\|_{H^4_\mu }\le \|Au\|_{L^2_\mu}+\|u\|_{L^2_\mu}\le C\|u\|_{H^4_\mu }$$
for all $u\in D(A)$ and some constant $C>0$.
\end{remark}

\section{The bi-Ornstein-Uhlenbeck operator}\label{bi-OU}
In this section we consider the Gaussian measure $\mu(x)=\gamma e^{-\frac{|x|^2}{2}}$ for $x\in \R^N$, where $\gamma =(2\pi)^{-N/2}$. In this case the operator $L$ is the classical symmetric Ornstein-Uhlenbeck operator
\begin{eqnarray*}
Lf &=& \Delta f-x\cdot \nabla f\\
f\in D(L)&=& H^{2}_\mu(\R^N).
\end{eqnarray*}
It is easy to see that $\mu$ satisfies all the assumptions of the previous sections. 
So, all the previous results can be applied to the above Ornstein-Uhlenbeck operator. In particular we deduce from \eqref{weighted_Rellich} and Theorem \ref{Th: Optimal} the following Rellich inequality for the Ornstein-Uhlenbeck operator.
\begin{theorem}\label{weighted-Rellich-OU}
Assume that $N\ge 5$ and $\mu(x)=\gamma e^{-\frac{|x|^2}{2}}$. Then for any $u\in H_\mu^2(\R^N)$ we have
\begin{eqnarray*}
(C_0-1)^2\int_{\R^N}\frac{|u(x)|^2}{|x|^4}d\mu 
&\leq & \int_{\R^N}|Lu(x)|^2d\mu+\frac{2(C_0-1)C_1}{C_0}\int_{\R^N}|\nabla u(x)|^2d\mu  \\
& &\quad +\frac{2(C_0-1)(C_1)^2}{C_0}\int_{\R^N}|u(x)|^2d\mu 
\end{eqnarray*}
for some positive constant $C_1$ and  $(C_0-1)^2=\left(\frac{N(N-4)}{4}\right)^2$ is the best constant.
\end{theorem}

Since we know without any restriction on $N$ that $D(L)=H^2_\mu(\R^N)$, 
one can, by direct computations, characterize the domain of the bi-Ornstein-Uhlenbeck operator for any $N\ge 1$. Moreover the following result shows that the corresponding semigroup is given by an explicit kernel.
\begin{theorem}\label{thm:bi-OU}
For any $N\ge 1$,
\begin{eqnarray*}
Af &=& \Delta^2 f -2x\cdot\nabla(\Delta  f) +Tr (x\otimes x D^2f)-2  \Delta f+ x\cdot\nabla f\\
f\in D(A)&=& H_\mu^4(\R^N)
\end{eqnarray*}
Moreover,
$$e^{-tA}f(x)=\int_{\R^N}k(t,x,y)f(y)\,dy,\quad x\in \R^N,\,t>0,\,f\in L_\mu^2(\R^N),$$
where
\begin{eqnarray*}
k(t,x,y) &=&(4\pi t)^{-\frac{1}{2}}\int_0^\infty e^{-\frac{s^2}{4t}}(p(is,x,y)+p(-is,x,y))\,ds\\
&=& \sqrt{2}(8\pi)^{-\frac{N+1}{2}} \int_0^\infty e^{-\frac{s^2}{4}}(\sin (s\sqrt{t}))^{-N/2}e^{-\frac{|e^{-is\sqrt{t}}x-y|^2}{8}}\cos\left(\frac{N}{2}(s\sqrt{t}-\frac{\pi}{2})+\frac{|e^{-is\sqrt{t}}x-y|^2}{8\tan(s\sqrt{t})}\right)ds
\end{eqnarray*}
for $t>0$ and $x,y\in \R^N$.
\end{theorem}
\begin{proof}
For the characterization of $D(A)$ we recall that $D(L)=H^2_\mu(\R^N)$ and hence
$$D(A)=\{u\in H_\mu^2(\R^N);\,Lu\in H_\mu^2(\R^N)\}.$$
An easy computation shows that
$ D_{i} \left( \frac{\nabla \mu}{\mu} \right)=-e_i$ and
$D_{ij} \left( \frac{\nabla \mu}{\mu} \right)=0$, where $e_i$ is the $i-th$ canonical vector of $\R^N$.

Thus, by \eqref{eq:dom-derive1} and \eqref{eq:dom-derive2}, we have 
$$
D_{k}Lu=LD_{k}u- D_ku
\hbox{\ and }
D_{hk}Lu=LD_{hk}u -2D_{hk}u
$$
for any $u\in H_\mu^4(\R^N)$.
So, $D_{k}Lu$ and $D_{hk}Lu$ belong to $L^2_\mu(\R^N)$ for all $h,\,k\in \{1,2,\ldots ,N\}$. Thus $Lu\in H_\mu^2(\R^N)$ and hence $H_\mu^4(\R^N)\subset D(A)$. 

For the other inclusion, since  $D(L)=H^2_\mu(\R^N)$, we have to prove that $D_{hk}u\in D(L)$ for any $h,k=1\,\ldots ,N$ whenever $u\in D(A)$. To this purpose let us consider $u\in D(A)$. As in the proof of Theorem \ref{th:characterization-2}, integrating by parts, we get
$$
\int_{\R^N}\nabla(D_{hk}u)\cdot \nabla \varphi \,d\mu = -\int_{\R^N}\left[D_{hk}Lu +2D_{hk}u\right] \varphi\,d\mu 
$$
for all $\varphi\in C_c^\infty(\R^N)$. 
So to conclude we  just have to observe that $D_{hk}Lu+2D_{hk}u\in L_\mu^2(\R^N)$.

For the last statement we recall (see \cite[Chapter 9]{ber-lor}, \cite{Lunardi}, \cite{MPRS}) that $L$ generates the analytic $C_0$-semigroup $T(t)$ of angle $\frac{\pi}{2}$ given by 
$$T(t)f(x)=(4\pi (1-e^{-2t}))^{-\frac{N}{2}}\int_{\R^N}e^{-\frac{|e^{-t}x-y|^2}{4(1-e^{-2t})}}f(y)\,dy,\quad t>0,\,x\in \R^N.$$
Moreover, it follows from \cite[Theorem 9.3.25]{ber-lor} that 
$$\|T(z)\|_{\mathcal{L}(L^2_\mu(\R^N))}\le 1,\quad \forall z\in \mathbb{C} \hbox{\ with }\Re z\ge 0.$$
Thus, applying \cite[Proposition 3.9.1 and Remark 3.9.3]{ABHN}, we deduce that $iL$ generates a $C_0$-group on $L^2_\mu(\R^N)$ given by
\begin{eqnarray*}
T(it)f(x) &=& \lim_{\varepsilon \to 0^+}(4\pi (1-e^{-2(\varepsilon +it)}))^{-\frac{N}{2}}\int_{\R^N}e^{-\frac{|e^{-(\varepsilon +it)}x-y|^2}{4(1-e^{-2(\varepsilon +it)})}}f(y)\,dy\\
&=& \int_{\R^N}p(it,x,y)f(y)\,dy,\quad t\in \R\setminus \pi \mathbb{Z} ,\,x\in \R^N , 
\end{eqnarray*}
where
\begin{equation}\label{eq:last-formule}
p(it,x,y)=
(4\pi (1-e^{-2it}))^{-\frac{N}{2}}e^{-\frac{|e^{-it}x-y|^2}{4(1-e^{-2it})}}, \quad t\in \R \setminus \pi\mathbb{Z},\,x,y\in \R^N.
\end{equation}
Thus, by \cite[Corollary 3.7.15]{ABHN},
the semigroup $(e^{-tA})$ generated by the bi-Ornstein-Uhlenbeck operator $-A=(iL)^2$ on $L_\mu^2 (\R^N)$  is given by
$$e^{-tA}f(x)=(4\pi t)^{-\frac{1}{2}} \int_0^\infty e^{-\frac{s^2}{4t}}(T(is)+T(-is))f(x)\,ds,\quad x\in \R^N,\,t>0,\,f\in L_\mu^2(\R^N).$$
Thus, $(e^{-tA})$ is given by a kernel $k$ with
$$
k(t,x,y) =(4\pi t)^{-\frac{1}{2}}\int_0^\infty e^{-\frac{s^2}{4t}}(p(is,x,y)+p(-is,x,y))\,ds
$$
for $t>0$ and $x,y\in \R^N$. So, the explicit formula of $k$ follows from \eqref{eq:last-formule}. This ends the proof of the theorem.

\end{proof}

To conclude the section we provide some examples of measures giving rise to a singular drift that satisfy our standing assumptions.

\begin{example}
\begin{enumerate}
\item 
Let us consider the density 
\begin{align*}
\mu(x)=C\frac{1+|x|^\alpha}{1+|x|^\beta}, \quad x\in \R^N,
\end{align*}
where $\alpha,\beta>0$, $\beta>\alpha+N$ and $C$ is a normalizing factor. These conditions  imply that $\mu\in L^1(\R^N)$ and $\mu(dx):=\mu(x)dx$ is a probability measure. Let us compute the gradient of $\mu$: we have
\begin{align*}
\nabla \mu 
= & C\alpha\frac{|x|^{\alpha-2}x}{1+|x|^\beta}-C\beta\frac{(1+|x|^\alpha)|x|^{\beta-2}x}{(1+|x|^\beta)^2}, \quad x\in \R^N\setminus\{0\},
\end{align*}
and
\begin{align*}
\Delta \mu
= &C\alpha\frac{(\alpha-2+N)|x|^{\alpha-2}}{1+|x|^\beta}
-2C\alpha\beta\frac{|x|^{\alpha+\beta-2}}{(1+|x|^\beta)^2} 
-C\beta(1+|x|^\alpha)\frac{(\beta-2+N)|x|^{\beta-2}}{(1+|x|^\beta)^2} \\
&+2C\beta^2(1+|x|^\alpha)\frac{|x|^{2\beta-2}}{(1+|x|^\beta)^3}, \quad x\in \R^N\setminus\{0\}.
\end{align*}
It follows that
\begin{align*}
\frac{\nabla \mu}{\mu} =\alpha\frac{|x|^{\alpha-2}x}{1+|x|^\alpha}-\beta\frac{|x|^{\beta-2}x}{1+|x|^\beta}, \quad x\in \R^N\setminus\{0\},
\end{align*}
and
\begin{align*}
\left|\frac{\Delta \mu}{\mu} \right|
\sim |\alpha(\alpha +N-2)||x|^{\alpha-2}, \quad |x|\rightarrow 0.
\end{align*}
Hence, assumptions $(H1)$ and $(H2)$ are verified. The associated Kolmogorov type operator is
\[Lu=\Delta u+\left(\alpha\frac{|x|^{\alpha-2}}{1+|x|^\alpha}-\beta\frac{|x|^{\beta-2}}{1+|x|^\beta}\right)x\cdot\nabla u .\]
Let us now prove $(H3)$. We notice that there exist $R_1,C>0$ such that
\begin{align}
\label{stima_example_1}
\left|\frac{\nabla \mu}{\mu} \right|\geq C\left({|x|^{\alpha-1}}\chi_{B_{R_1}}(x)+\frac{|x|^{\beta-1}}{1+|x|^\beta}\chi_{B^c_{R_1}}(x)\right), \quad x\in\R^N\setminus \{0\}.
\end{align}
It is not hard to see that there exist $K_0,C_0>0$ such that for any $i,j=1,\ldots,N$ we have
\begin{align}
\label{stima_example_2}
\left|D_i\left(\frac{\nabla \mu}{\mu}\right) \right|
& \leq C_0|x|^{\alpha-2}\chi_{B_{R_1}}(x)+K_0\frac{1}{|x|^2}\chi_{B^c_{R_1}}(x), \quad x\in \R^N\setminus\{0\},\\
\left|D_{ij}\left(\frac{\nabla \mu}{\mu}\right) \right|
& \leq C_0|x|^{\alpha-3}\chi_{B_{R_1}}(x)+K_0\frac{1 }{|x|^3}\chi_{B^c_{R_1}}(x), \quad x\in \R^N\setminus\{0\}.
\label{stima_example_3}
\end{align}
From \eqref{stima_example_1}, \eqref{stima_example_2} and \eqref{stima_example_3} it follows that also $(H3)$ is satisfied.

\item Consider $\mu=\kappa e^{-|x|^m}$, where $m>0$ and $\kappa$ is a normalising factor. 
By simple computations one has $\mu\in H^1_{\rm loc}(\R^N)$ and
\[\frac{\nabla\mu}{\mu}=-m|x|^{m-2}x,\qquad x\in\R^N\setminus\{0\},\]
so that $\mu$ satisfies Hypothesis $(H1)$. Then, the associated Kolmogorov type operator $L$ is
\[Lu=\Delta u-m|x|^{m-2}x\cdot\nabla u.\]
Furthermore, $\Delta\mu\in L^1_{\rm loc}(\R^N)$ and 
\[U=-\frac14 m^2|x|^{2m-2}+\frac12 m(N+m-2)|x|^{m-2},\qquad x\in\R^N\setminus\{0\}.\]
Hence, Hypothesis $(H2)$ is satisfied. 
Finally,
\begin{align*}\left|D_i\left(\frac{D_j\mu}{\mu}\right)\right|&\leq C|x|^{m-2},\qquad x\in\R^N\setminus\{0\},\\
\left|D_{ij}\left(\frac{D_k\mu}{\mu}\right)\right|&\leq C_1|x|^{m-3},\qquad x\in\R^N\setminus\{0\}.
\end{align*}
Therefore, also Hypothesis $(H3)$ is fulfilled. 

As a consequence, also the measure in Example \ref{ex:measure} satisfies $(H1),(H2)$ and $(H3)$.

\end{enumerate}
\end{example}

\providecommand{\bysame}{\leavevmode\hbox to3em{\hrulefill}\thinspace}
\providecommand{\MR}{\relax\ifhmode\unskip\space\fi MR }
\providecommand{\MRhref}[2]{%
  \href{http://www.ams.org/mathscinet-getitem?mr=#1}{#2}
}
\providecommand{\href}[2]{#2}


\begin{thebibliography}{1}

\bibitem{ada} D. Adams,\emph{ $L^p$ potential theory techniques and nonlinear PDE,} Potential theory (Nagoya, 1990), 1–15, de Gruyter, Berlin, 1992. 



\bibitem{alb-lor-man}
A.~Albanese, L.~Lorenzi, E.~Mangino, \emph{$L^p$–uniqueness for elliptic
  operators with unbounded coefficients in $\R^N$}, J. Funct. Anal. \textbf{256} (2009), no. 4, 
  1238--1257.



\bibitem{ant} S.S. Antman,  \emph{Nonlinear problems of elasticity,} Second edition. Applied Mathematical Sciences, 107. Springer, New York, 2005. 




 
\bibitem{ABHN}
W. Arendt, C.J.K. Batty, M. Hieber, F. Neubrander, \emph{Vector-valued {L}aplace {T}ransforms and {C}auchy {P}roblems}, Birkh\"auser Verlag, 2001.


\bibitem{aro} S. Arora, \emph{Locally eventually positive operator semigroups}, preprint, 	arXiv:2101.11386.



\bibitem{ber-lor}
M.~Bertoldi and L.~Lorenzi, \emph{Analytical methods for {M}arkov semigroups},
  Chapman \& Hall/CRC, 2007.

\bibitem{bog_kry_roc01}
V.~I. Bogachev, N.~V. Krylov, and M.~R\"{o}ckner, \emph{On regularity of
  transition probabilities and invariant measures of singular diffusions under
  minimal conditions}, Comm. Partial Differential Equations \textbf{26} (2001),
  no.~11-12, 2037--2080. \MR{1876411}

\bibitem{can-gre-rha-tac}
A.~Canale, F.~Gregorio, A.~Rhandi, C.~Tacelli, 
\emph{{W}eighted {H}ardy's inequalities and {K}olmogorov-type operators},
J Appl. Anal. \textbf{98} (2019), no.~7, 1236--1254.


\bibitem{dap}  G. Da Prato, \emph{Regularity results for Kolmogorov equations on $L^2(H,\mu)$ spaces and applications,} Ukrainian
Math. J. \textbf{49 } (3) (1997), 448–457.

\bibitem{dap-ves} G. Da Prato, V. Vespri,
\emph{Maximal $L^p$ regularity for elliptic equations with unbounded coefficients,}
Nonlinear Anal. \textbf{49} (2002), no. 6, Ser. A: Theory Methods, 747–755.

\bibitem{dan-glu-ken1} D. Daners, J. Glück, J.B. Kennedy,   \emph{Eventually and asymptotically positive semigroups on Banach lattices,} J. Differential Equations \textbf{261} (2016), no. 5, 2607–2649.

\bibitem{dan-glu-ken2} D. Daners, J. Glück, J.B. Kennedy,  \emph{Eventually positive semigroups of linear operators,} J. Math. Anal. Appl. \textbf{433} (2016), no. 2, 1561–1593.

\bibitem{EnNa00}
K.J.~Engel, R.~Nagel,
\newblock{One-parameter semigroups for linear evolution equations, Graduate Texts in Mathematics},
\newblock{Springer-Verlag, New York, 2000}.

\bibitem{fer} L.C.F. Ferreira, V.A. Ferreira, 
\emph{On the eventual local positivity for polyharmonic heat equations,} (English summary)
Proc. Amer. Math. Soc. \textbf{147} (2019), no. 10, 4329–4341.

\bibitem{fer-gaz-gru} A. Ferrero, F. Gazzola, H.-C. Grunau,   \emph{Decay and eventual local positivity for biharmonic parabolic equations,} Discrete Contin. Dyn. Syst. \textbf{21} (2008), no. 4, 1129–1157.




\bibitem{gaz-gru} F. Gazzola, H.-C. Grunau,  \emph{Eventual local positivity for a biharmonic heat equation in $\R^n$,} Discrete Contin. Dyn. Syst. Ser. S \textbf{1} (2008), no. 1, 83–87.



\bibitem{gre-mil} F. Gregorio, S. Mildner,  \emph{Fourth-order Schr\"odinger type operator with singular potentials,} Arch. Math. (Basel) \textbf{107} (2016), no. 3, 285–294. 


\bibitem{gre-mug}
F.~Gregorio, D.~Mugnolo, 
\emph{Bi-Laplacians on graphs and networks}, 
J. Evol. Equ. \textbf{20} (2020), no. 1, 191--232. 


\bibitem{hoc} K.J. Hochberg,   \emph{A signed measure on path space related to Wiener measure,} Ann. Probab. \textbf{ 6} (1978), no. 3, 433–458.

\bibitem{Ho81}
J.G.~Hooton,
\emph{Compact {S}obolev imbeddings on finite measure spaces},
{J. Math. Anal. Appl.}, \textbf{83} (1981), 2, {570--581}.

\bibitem{lor-16}
L. Lorenzi,
\emph{Analytical methods for Kolmogorov equations. Second Edition},
\newblock{CRC Press, Boca Raton, FL, 2017}.

\bibitem{Lunardi}
A. Lunardi, \emph{On the Ornstein-Uhlenbeck operator in $L^2$ spaces with respect to invariant measures}, Trans. Amer. Math. Soc. \textbf{349} (1997), no. 1, 155-169.




\bibitem{mel} V.V. Meleshko, \emph{Selected topics in the history of the two-dimensional bi-harmonic problem,} Appl. Mech. Rev. \textbf{56} (2003), no. 1, 33-85.


\bibitem{MPRS}
G. Metafune, J. Pr\"uss,   A. Rhandi, R. Schnaubelt, \emph{The domain of the Ornstein-Uhlenbeck operator on an $L^p$-space with invariant measure}, Ann. Sc. Norm. Sup. Pisa Cl. Sci. (5), \textbf{1} (2002), no. 2, 471-485.


\bibitem{ouh}
E. M. Ouhabaz, \emph{Analysis of heat equations on domains}, London Mathematical
Society Monograph Series, 31. Princeton University Press,
Princeton, NJ, 2005.

\bibitem{pru-rha-sch} J. Prüss, A. Rhandi, R. Schnaubelt,   \emph{The domain of elliptic operators on $L^p(\R^d)$ with unbounded drift coefficients,} Houston J. Math. \textbf{32} (2006), no. 2, 563–576. 

\bibitem{sed} T. Sedrakyan, L. Glazman, A. Kamenov, \emph{Absence of Bose condensation on lattices with moat bands,} Phys. Rev. B \textbf{89} (2014), 201112 (R).

\bibitem{skr} I. V. Skrypnik,  \emph{ Methods for analysis of nonlinear elliptic boundary value problems,} Translated from the 1990 Russian original by Dan D. Pascali. Translations of Mathematical Monographs, 139. American Mathematical Society, Providence, RI, 1994. 

\bibitem{Th98}
H.R.~Thieme, 
\emph{Balanced exponential growth of operator semigroups}, J. Math. Anal. Appl., (223), \textbf{1},  (1998), 30-49.

\end{thebibliography}

\end{document}